\documentclass[11pt]{amsart}

\usepackage{epigamath-voisin}

\usepackage[english]{babel}

\numberwithin{equation}{section}

\usepackage{enumitem}

\usepackage{tikz}  
\usepackage{color} 
\usetikzlibrary{arrows,calc}
\usepackage{graphicx}
\usetikzlibrary{decorations.pathmorphing}
\usetikzlibrary{decorations.markings} 
\usetikzlibrary{decorations.pathreplacing} 
\usetikzlibrary{plothandlers}
\usepackage{diagrams}
\usepackage[matrix,arrow,curve]{xy}
\usepackage{MnSymbol}

\newtheorem{theorem}{Theorem}[section]
\newtheorem{corollary}[theorem]{Corollary}
\newtheorem{lemma}[theorem]{Lemma}
\newtheorem{proposition}[theorem]{Proposition}

\newtheorem{conjecture}[theorem]{Conjecture}

\theoremstyle{definition}

\newtheorem{definition}[theorem]{Definition}

\theoremstyle{remark}

\newtheorem{example}[theorem]{Example}
\newtheorem{remark}[theorem]{Remark}

\newcommand{\supth}[1]{\ensuremath{#1^{\mathrm{th}}}}
\newcommand{\supnd}[1]{\ensuremath{#1^{\mathrm{nd}}}} 

\newcommand{\FF}{{\mathbb F}}
\newcommand{\GG}{{\mathbb G}}
\newcommand{\PP}{{\mathbb P}}

\newcommand{\TT}{{\mathbb T}}

\newcommand{\ZZ}{{\mathbb Z}}

\newcommand{\ann}{{\rm{ann}}}
\newcommand{\Ext}{{\rm{Ext}}}

\newcommand{\id}{{\rm{id}}}

\newcommand{\coker}{{\rm{coker}\,}}

\newcommand{\ev}{\mathrm{ev}}
\newcommand{\odd}{\mathrm{odd}}
\newcommand{\longleftmapsto}{\mathrel{\reflectbox{\ensuremath{\longmapsto}}}}
\newcommand{\es}{\mathrm{es}}
\newcommand{\bigslant}[2]{\left.\raisebox{.2em}{$\displaystyle #1$}\middle/\raisebox{-.2em}{$\displaystyle #2$}\right.}

\newcommand{\s}{\mathcal}
\newcommand{\sA}{{\s A}}
\newcommand{\sB}{{\s B}}
\newcommand{\sC}{{\s C}}

\newcommand{\sE}{{\s E}}
\newcommand{\sF}{{\s F}}
\newcommand{\sG}{{\s G}}
\newcommand{\sH}{{\s H}}

\newcommand{\sL}{{\s L}}

\newcommand{\sN}{{\s N}}
\newcommand{\sO}{{\s O}}

\newcommand{\sU}{{\s U}}

\newcommand{\cO}{{\s O}}

\newcommand{\cN}{{\s N}}

\newcommand{\lra}{\longrightarrow}

\newcommand{\tensor}{\otimes}

\newcommand{\punkt}{\hspace{-.3ex}\raise.15ex\hbox to1ex{\Huge.}}

\DeclareMathOperator{\Pic}{Pic}

\DeclareMathOperator{\Sym}{Sym}

\DeclareMathOperator{\Spec}{Spec}

\DeclareMathOperator{\Hom}{Hom}

\DeclareMathOperator{\codim}{codim}
\DeclareMathOperator{\rank}{rk}

\newcommand{\gm}{\mathfrak m}

\def\AA{{\mathbb A}}

\def\PP{{\mathbb P}}

\def\FF{{\mathbb F}}
\def\TT{{\mathbb T}}

\def\sEnd{{\mathcal End}}

\DeclareMathOperator{\Cliff}{{\rm Cliff}}
\DeclareMathOperator{\chara}{{char}}
\DeclareMathOperator{\hess}{{hess}}

\def\all{{\{1,\ldots,2g+2\}}}

\def\Spec{{{\rm Spec}\,}}

\def\lbracket{{[\kern-1.5pt[}}
\def\rbracket{{]\kern-1.5pt]}}

\YearArticle{2025} \EpigaArticleNr{21} \ReceivedOn{December 15, 2022}
\InFinalFormOn{May 21, 2024}
\AcceptedOn{January 3, 2025}

\title{Hyperelliptic Curves and Ulrich sheaves on the complete intersection of two quadrics}
\titlemark{Hyperelliptic Curves and Ulrich sheaves on the complete intersection of two quadrics}

\author{David Eisenbud}
\address{Department of Mathematics, University of California at Berkeley and the Mathematical
Sciences Research Institute, Berkeley, CA 94720, USA}
\email{de@msri.org}

\author{Frank-Olaf Schreyer}
\address{Fachbereich Mathematik, Universit\"at des Saarlandes, Campus E2 4, D-66123 Saar\-br\"ucken, Germany}
\email{schreyer@math.uni-sb.de}

\authormark{D.~Eisenbud and F.-O.~Schreyer}

\AbstractInEnglish{Using the connection between hyperelliptic curves, Clifford algebras, and smooth complete intersections $X$ of two quadrics, we describe
Ulrich bundles on $X$ and construct some of minimal possible rank.}



\acknowledgement{The first author was partially supported by NSF grant 2001649. 
This work is also a contribution to Project-ID 286237555 -- TRR 195 -- by the Deutsche Forschungsgemeinschaft (DFG, German Research Foundation) of the second author.
Part of the work was done while the authors were supported by the Simons Laufer Mathematical Sciences Institute.}

\dedication{For Claire Voisin on the occasion of her Birthday}

\begin{document}



\maketitle

\begin{prelims}

\DisplayAbstractInEnglish

\bigskip

\DisplayKeyWords

\medskip

\DisplayMSCclass







\end{prelims}


\newpage

\setcounter{tocdepth}{1}

\tableofcontents


\section{Introduction}

Let $k$ be an algebraically closed field of characteristic not 2. 
The periodicity theorem of Kn\"orrer \cite{Knorrer} shows that the indecomposable Ulrich
bundles on a smooth quadric hypersurface in $\PP^{2g+1}$ over $k$
have rank $2^{g-1}$. In this paper we construct Ulrich bundles of the same rank $2^{g-1}$
on every smooth complete intersection $X$ of 2 quadrics in $\PP^{2g+1}$, and we show that
every Ulrich bundle has rank of the form $r2^{g-2}$ where $r\geq 2$ and $rg$ is  even.
To prove this we use an equivalence of categories that extends Reid's famous description of the
Jacobian of a hyperelliptic curve \cite{ReidThesis}.

Let $X\subset \PP^n$ be a projective scheme with homogenous coordinate ring $P_X$. Recall that a  sheaf $\sE$ on $X$  is called Ulrich if the graded module of twisted global sections  $H^0_*(\sE)$ is a maximal Cohen--Macaulay $P_X$-module generated in degree 0 and having linear free resolution over the coordinate ring of $\PP^n$, or equivalently if
$H^i(\sE(m)) = 0$ for all $m$ with  $-1\ge m \ge -\dim X$ and all $i$. See \cite{ES1} for further information and examples.

Let $X$ be the smooth complete intersection defined by two quadratic forms $q_1, q_2$ on 
$\PP^{2g+1}$ over an algebraically closed field $k$ of characteristic not 2.

The pencil of quadrics $sq_{1}+tq_{2}, (s,t)\in \PP^{1}$ becomes singular at $2g+2$ points of $\PP^{1}$.
 Let $E$ be the hyperelliptic curve with homogeneous coordinate ring $k[s,t,y]/(y^{2} - f)$ branched over these points, and let $C$ be the $\ZZ$-graded Clifford algebra of the form $sq_1+tq_2$ over
$k[s,t]$. 

We give two approaches to the construction of Ulrich sheaves on $X$. The first makes use of three categories:
\begin{enumerate}
\item[(i)] the category of coherent sheaves on $E$,
\item[(ii)] the category of graded $C$-modules, and
\item[(iii)] the category of coherent sheaves on $X$.
\end{enumerate}
Categories (i) and (ii) are related by Morita equivalence, while categories (ii) and (iii) are related by a version 
of the Bernstein--Gel'fand--Gel'fand correspondence.

Composing these correspondences to go from (i) to (iii), we show that every Ulrich module on $X$ has rank $r2^{g-2}$ for
some integer $r\geq 2$. 

Following \cite{MR2639318} we say that a bundle $\sB$ on $E$ has the \emph{Raynaud property} if $H^0(C, \sB) = H^1(C,\sB)=0$. We use the fact that
the center of the even Clifford algebra is the homogeneous coordinate ring of $E$, and that  the category of coherent sheaves of modules over the sheafified even Clifford algebra $\sC^{\ev}\cong \sEnd_E(\sF_U)$ is Morita equivalent to the category of
coherent sheaves on $E$ \emph{via} an $\sO_{E}-\sC^{\ev}$ bundle $\sF_U$ defined in Section~\ref{pencil section}.  With this notation, our first main theorem is the following.

\begin{theorem}\label{main theorem}
There is a 1-1 correspondence between Ulrich bundles on the smooth complete intersection of two quadrics $X\subset \PP^{2g+1}$ and  bundles
of the form
$\sG \otimes \sF_{U}$ with the Raynaud property on the corresponding hyperelliptic curve $E$ .
The Ulrich bundle corresponding to a rank $r$ vector bundle $\sG$ has rank $r2^{g-2}$. 

If $\sL$ is a line bundle on $E$ then $\sL \otimes \sF_{U}$ does not have the Raynaud property, so the minimal possible rank of an Ulrich sheaf on $X$ is $2^{g-1}$, and Ulrich bundles of  rank $2^{g-1}$ exist.
\end{theorem}

The set of bundles $\sG$ such that $\sG \otimes \sF_U$ has the Raynaud property forms a (possibly empty) open subset in any flat family of  rank $r$ vector bundles on $E$. Our second main theorem,  the existence statement for $r=2$ is proven using a previously undiscovered property of Kn\"orrer's matrix factorizations to give a  construction of an Ulrich sheaf of
the minimal possible rank, $2^{g-1}$ on any smooth complete intersection
of two quadrics in $\PP^{2g+1}$ and in $\PP^{2g+2}$. 

Based on computed examples using our package \cite{EKS} with Yeongrak Kim, we conjecture the following.

\begin{conjecture} There exist indecomposable Ulrich bundles of rank $r 2^{g-2}$ on every smooth complete intersection of two quadrics in $\PP^{2g+1}$ for $g\ge 1$ and $r \ge 2$ if and only if $r g \equiv 0 \mod 2$.
\end{conjecture}
\noindent
By Proposition \ref{congruenceCondition} the condition is necessary.

\medskip
In Section~\ref{sec:vector bundles} we explain the description of vector bundles on $E$ in terms of matrix factorizations. In the case of line bundles, this theory can be traced through Mumford's~\cite{Mumford-Tata} to work of Jacobi~\cite{Jacobi}.
 
In Section~\ref{BGG} we explain the relation of categories (ii) and (iii), a form of the Bernstein--Gel'fand--Gel'fand (BGG) correspondence that holds for all complete intersections of quadrics. As far as we know this correspondence was first introduced in~\cite{BEH}, and greatly extended in~\cite{Kapranov}. For the reader's convenience we review the results that we will use.
 
In Section~\ref{pencil section} we establish the Morita equivalence between categories (i) and (ii). In fact every maximal (simultaneous) isotropic plane $U$ for $q_1$ and $q_2$ gives rise to a Morita bundle $\sF_{U}$ and any two differ by the tensor product with a line bundle on $E$. This explains the well-known result of Miles Reid's thesis that the space of maximal (simultaneous) isotropic planes for $q_1$ and $q_2$ can be identified with  the Jacobian of $E$.
 
In Section~\ref{sec:Tate} we put these tools together with the theory of Tate resolutions and maximal Cohen--Macaulay approximations to establish the equivalence between Ulrich modules of rank $r2^{g-2}$ on $X$ and vector bundles of  rank $r$
 on $E$ that satisfy certain cohomological conditions. We show that no line bundles on $E$ satisfy the conditions, establishing
the lower bound for the rank of Ulrich modules announced above. 
This section was inspired by Buchweitz's famously unpublished manuscript on Koszul duality from 1986, now available at \cite{Buchweitz} and by the theory of Cohen--Macaulay approximations by Auslander and Buchweitz~\cite{Auslander-Buchweitz}.

It is natural to look for Ulrich bundles on $X$ using the shape of their  Tate resolutions over $P_{X}$. Theorem \ref{construction of TN} is analogous to the main result on Tate resolution of coherent sheaves on $\PP^{n}$ in \cite{EFS}: the Betti table of the Tate resolution over the exterior algebra  coincides with cohomology tables of the corresponding sheaf. In Theorem~\ref{construction of TN} the resolution over the exterior algebra is replaced by the Tate resolution over $P_{X}$.
  
In Section~\ref{sec:Ulrich}, which is independent of the rest of the paper, we give a direct construction of Ulrich modules of rank $2^{g-1}$ on any smooth complete intersection of quadrics in $\PP^{2g+1}$ and $\PP^{2g+2}$ with the minimal possible rank, $2^{g-1}$. In the case $g=2$ the existence and minimality was established by~\cite{ChoKimLee} with a different method.

\subsection*{Historical remarks}
The study of complete intersections of quadrics has a long history. 
The connection to vector bundles was discovered by Newstead \cite{Newstead}, Reid \cite{ReidThesis} and Desale--Ramanan~\cite{Desale-Ramanan} in the 1970's. The connection with Clifford algebras and Koszul pairs was used in~\cite{BEH} and more generally by Kapranov~\cite{Kapranov} in the 1980's.

The first three sections of the paper, which take the point of view of matrix factorizations, have their roots in an unpublished manuscript by our dear friend Ragnar Buchweitz (1952--2017) and the second author in 90's, now lost. The referee kindly pointed out to us that parts of Theorem~\ref{main1} can be deduced from Kuznetsov's work~\cite{Kuznetsov}, which, like the work of Kapranov~\cite{Kapranov} instead takes the point of view of derived categories.

The theory of quadratic complete intersections has many guises, and appears in descriptions of certain completely integrable systems, for example in the recent paper of Claire Voisin and her coauthors~\cite{BEHLV}. 

\subsection*{Acknowledgements}
We are grateful to the referees for their many constructive comments. This paper would have been impossible without the program Macaulay2~\cite{M2}. We thank Yeongrak Kim for working with us on the associated Macaulay2 package, distributed with the Macaulay2 program.

\section{Vector Bundles over a hyperelliptic curve \emph{via} matrix factorizations}\label{sec:vector bundles}

Let $E$ be a hyperelliptic curve of genus $g$ and let $\pi\colon E\to \PP^{1}$ its  double cover of $\PP^{1}$.
Let $\sH = \pi^{*}\cO_{\PP^{1}}(1)$ and
let $f(s,t)$ be the homogeneous polynomial of degree $2g+2$ such that
\[
R_{E}:= k[s,t,y]/(y^{2} - f) = \bigoplus_{n}H^{0}(E, \sH^{\otimes n}),
\]
so that the roots of $f$ are the ramification points of $\pi$ and $y\in H^{0}(E, \sH^{\otimes g+1})$.

For a coherent sheaf $\sG$ on $E$ we denote by
\[
H^i_*(\sG)= \bigoplus_n H^i(E,\sG\otimes \sH^{\otimes n}).
\]
Thus $H_*^0(\sO_E)=R_E$ and $\pi_*$ corresponds to forgetting the $y$-action on $H^0_*(\sG)$.

\begin{proposition}\label{vectorBundlesOnE}
If $\sL$ is a vector bundle on $E$, then $B = H^{0}_*(\sL)$ is a graded
free module over the homogeneous coordinate ring $k[s,t]$ of $\PP^{1}$, and $y\colon \sL \to \sL(g+1)$ induces a map 
$\phi = H^{0}_{*}y\colon  B \to B(g+1)$ such that
$\phi^{2}$ is multiplication by $f$; that is, a matrix factorization of $f$. 

Furthermore, given a graded free module $B$ corresponding to the vector bundle $\sB$ on $\PP^{1}$, and a map \mbox{$\phi\colon B\to B(g+1)$} with $\phi^{2} = f\cdot Id_{B}$, the sheaf 
\[
\sL=\coker\left(y-\phi \colon\, \pi^*\sB(-g-1) \lra \pi^* \sB\right)
\]
is a vector bundle on $E$
whose pushforward is $\sB$, and on which $y$ induces the matrix factorization $\phi$. We have 
\[
\chi(\sB) = \chi(\sL),\quad \rank \sB = 2\rank \sL, \quad \text{and}\quad \deg \sB=\deg \sL  - (\rank \sL)(1+g).
\]
\end{proposition}

The proof could be extended to show that
the category of vector bundles on $E$ is equivalent to the category of matrix factorizations of
$f$ over $k[s,t]$, \emph{cf.}~\cite{Eisenbud80}.

\begin{proof}[Proof of Proposition \ref{vectorBundlesOnE}]
The equation $\phi^{2} = f$ follows from functoriality. Conversely, if 
a matrix factorization $\phi^{2} = f\cdot Id_{B}$ is given, then $(y-\phi, y+\phi)$ is a matrix factorization of $y^{2}-f$ over $k[s,t,y]. $ Thus  the module $\coker(y-\phi)$ is a maximal Cohen--Macaulay $R_{E}$-module, and it follows that the sheaf
associated to its cokernel 
is a vector bundle  on $E$. 
\end{proof}

The next Theorem reduces the computation of the tensor product of vector bundles on $E$
to a syzygy computation, and will be used this way in the sequel.

\begin{theorem}\label{VB=MF}
If $\sL_{1}, \sL_{2}$ are vector bundles on $E$ with matrix factorizations $\phi_{i}$ on the graded free $k[s,t]$-modules
$
B_{i} = H_*^{0}(\sL_{i}),
$
then
\[
H_*^{0}(\sL_{1}\otimes \sL_{2}) 
= \ker\big(\phi_{1}\otimes 1 - 1\otimes \phi_{2}\colon\,
B_{1}\otimes B_{2}(g+1) \lra B_{1}\otimes B_{2}(2g+2)\big)
\]
and $\pi_{*}y$ acts on $\pi_{*}(\sL_{1}\otimes \sL_{2})$ with the common action of $\phi_{1}\otimes 1$ and $1\otimes \phi_{2}$.
\end{theorem}

\begin{proof} 
The following sequence of maps  is a complex because $y^{2} = f$:
\begin{equation}\tag{$*$}\label{eq:complex}
B_{1}\otimes B_{2}(-g-1) \rTo^{\phi_{1}\otimes 1 - 1\otimes \phi_{2}}  B_{1}\otimes B_{2}  \rTo^{\phi_{1}\otimes 1 + 1\otimes \phi_{2}}  B_{1}\otimes B_{2}(g+1) \rTo^{\phi_{1}\otimes 1 - 1\otimes \phi_{2}} B_{1}\otimes B_{2}(2g+2)
\end{equation}
Since the $k[s,t]$-module 
\[
\ker\left( 
B_{1}\otimes B_{2}(g+1) \rTo^{\phi_{1}\otimes 1 - 1\otimes \phi_{2}} B_{1}\otimes B_{2}(2g+2)\right)
\]
 is a \supnd{2} syzygy, it is free. Thus, to prove the theorem, it suffices to show that the complex~\eqref{eq:complex} is 
locally exact and that the sheaf cokernel
\[
\coker\left( \sB_{1}\otimes \sB_{2}(-g-1) \rTo^{\phi_{1}\otimes 1 - 1\otimes \phi_{2}} \sB_{1}\otimes \sB_{2}\right)
\]
is $\pi_*(\sL_1\otimes_E\sL_2)$.

For simplicity of notation we ignore the twists by powers of $\sH$. 
Note that  $\sB_{i}:=\pi_{*}(\sL_{i})$ is the sheafification
of $B_{i}$. Since $\sL_{i}$ is the cokernel of $y-\phi_{i}$ we see that
$\sL_{1}\otimes_{E}\sL_{2}$ is the cokernel of 
\[
\left( \pi^* \sB_{1}\otimes_{E} \pi^* \sB_{2}\right) \oplus \left(\pi^* \sB_{2}\otimes_{E} \pi^* \sB_{1}\right) \rTo^{(y\otimes 1-\phi_1\otimes 1 , 1\otimes y-1\otimes \phi_{2})} \pi^* \sB_{1}\otimes_{E} \pi^* \sB_{2}.
\]
Since the tensor products are over $E$, the maps $y\otimes 1$ and $1\otimes y$ are
equal, and are simply multiplication by $y$, so this says that $\sL_{1}\otimes \sL_{2}$
is the universal quotient of $\pi^{*}\sB_{1}\otimes_{E} \pi^* \sB_{2}$ on which the maps
$y, \phi_{1}\otimes 1,  1\otimes \phi_{2}$ all agree. Furthermore,
\[
\pi_{*}(\pi^{*}\sB_{1}\otimes_{E} \pi^* \sB_{2}) = \pi_{*}\pi^{*}(\sB_{1}\otimes_{\PP^{1}} \sB_{2})
= \pi_{*}(\sO_{E}) \otimes_{\PP^{1}} \sB_{1}\otimes_{\PP^{1}} \sB_{2}.
\]
where the action of $y$ is on the first factor only.
Thus $\pi_*(\sL_{1}\otimes \sL_{2})$
is the cokernel of 
\[
\phi_{1}\otimes 1 - 1\otimes \phi_{2} : \sB_{1} \otimes \sB_{2} \rTo \sB_{1} \otimes \sB_{2}.
\]

To complete the proof we must show that the sequence~\eqref{eq:complex} is locally exact.
Choose a point
$x\in \PP^{1}$ and denote the local ring $\sO_{\PP^{1},x}$ by $A$ and the $A$-module
$\sB_{i,x}$ by $F_{1}+yF_{1}$
 where the $F_{i}$ are free $A$-modules. The endomorphism $\phi_i$ takes
$F_{i}$ to $yF_{i}$ by multiplying with $y$, and
  $yF_{i}$ to $F_{i}$ by sending $y$ to $f \in A$. In this notation, the maps
 $\phi_{1}\otimes 1 \pm 1\otimes \phi_{2}$ may be written as block matrices of the form
\[
\bordermatrix{
                  &F_{1}\otimes F_{2}& F_{1}\otimes yF_{2}&yF_{1}\otimes F_{2}&yF_{1}\otimes yF_{2}\cr
 F_{1}\otimes F_{2}& 0              &    \pm f                     & f                             & 0\cr
 F_{1}\otimes yF_{2}& \pm 1     & 0                             & 0                            &f \cr
 yF_{1}\otimes F_{2}& 1            & 0                             & 0                            & \pm f \cr
 yF_{1}\otimes yF_{2}& 0          &    1                           &\pm 1                             & 0
 }
\]
Modulo the maximal ideal of $A$ both these maps have rank equal to twice the rank of $F_{1}\otimes F_{2}$, so the sequence above is locally split exact, as required.
\end{proof}

\begin{definition}\label{LI}
Let $f(s,t) = \prod_{i=1}^{2g+2}f_{i}$ be a factorization of $f$ into (necessarily distinct) linear factors, and, for $I\subset \{1,\ldots, 2g+2\}$,
write $f_{I} := \prod_{i\in I}f_{i}$. We write $\phi_{I}$ for the matrix
\[
\begin{pmatrix}
 0& f_{I^{c}} \\
 f_{I}&0\\
\end{pmatrix}\colon\, \sO_{\PP^1}\left(\lceil - |I|/2\rceil\right) \oplus \sO_{\PP^1}\left(\lceil -|I^{c}|/2\rceil\right) \lra \sO_{\PP^1}\left(\lceil |I^{c}|/2\rceil\right) \oplus \sO_{\PP^1}\left(\lceil |I|/2\rceil\right)
\]
on $\PP^1$ where $I^{c}$ denotes the complement of $I$. Note that $(\phi_{I}, \phi_{I^{c}})$ is a matrix
factorization of $f$. Let $\sL_{I}$ be the corresponding line bundle on $E$, as defined in Proposition~\ref{vectorBundlesOnE}.
Note that $\sL_{I} \cong \sL_{I^{c}}$ and $\sL_{\emptyset}\cong \sO_{E}$.
Write $I \Delta J= (I \setminus J) \cup (J \setminus I)$ for the symmetric difference of $I$ and $J$.
\end{definition}

\begin{theorem}\label{even and odd}
    For $I,J \subset \{1,\ldots,2g+2\}$
\[
    \sL_{I} \tensor \sL_{J} \cong 
    \begin{cases}
     \sL_{I \Delta J} & \ \text{if}\ |I| \cdot |J| \equiv 0 \! \mod \; 2, \\
     \sL_{I \Delta J}(\sH) & \ \text{else.}
     \end{cases}
\]
Thus the line bundles $\sL_{I}$ with $|I|$  even are the $2^{2g}$ two-torsion line bundles on $E$. The line bundles $\sL_{I}$ with $|I|$ odd are the $2^{2g}$  square roots of $\sO_{E}(\sH)$.
\end{theorem}

\begin{proof} In this case the matrix $\phi_{I}\otimes 1 - 1\otimes \phi_{J}$ has the form
\[
\begin{pmatrix} 
0 & f_{I^c} &-f_{J^c} &0\cr 
f_I & 0 &  0 &-f_{J^c} \cr
-f_J&0 & 0 &f_{I^c} \cr
0 & -f_J & f_I & 0 
\end{pmatrix}.
\]
By Theorem~\ref{VB=MF}, its kernel is the free module $H^{0}_{*}(\sL_{I}\otimes \sL_{J}).$ Because $J^c \setminus I^c = I \setminus J$ and $I \setminus J^c=J \setminus I^c$ this kernel
 contains the free submodule $B$ generated by the column vectors
\[
\begin{pmatrix} 
0 & f_{J^c\setminus I}\cr 
f_{I\setminus J} & 0  \cr
f_{J\setminus  I}&0  \cr
0 &  f_{I \setminus J^c}
\end{pmatrix}.
\]
These columns generate the kernel because the $2\times 2$ minors of $B$
have no common factor (see \cite[Corollary~1]{Buchsbaum-Eisenbud}).
 
To show that $\sL_{I}\otimes \sL_{J} \cong \sL_{I\Delta J}$ it now suffices to show that the matrix
representing the action of $\phi_{1}\otimes 1$ restricted to the columns of $B$ is
\[
\begin{pmatrix}
 0& f_{(I\Delta J)^{c}} \\
 f_{I\Delta J}&0
\end{pmatrix}.
\]

This, in turn, follows at once from the identities
\[ 
I^c \cup (I \setminus J) = (I \Delta J) \cup (J^c \setminus I), \quad I \cup (J \setminus I) = (I \Delta J) \cup (I \setminus J^c)
\]
and similarly
\[
I \cup (J^c \setminus I) = (I \Delta J)^c \cup (I\setminus J),\quad I^c \cup (I \setminus J^c)=(I \Delta J)^c \cup (J\setminus I).
\]

To show that 
 $\sL_I \not\cong \sL_J$ for $J \notin \{I,I^c\}$ are non-isomorphic, we consider the ideals generated by the entries of 
\[
\begin{pmatrix}
 0& f_{I^{c}} \\
 f_{I}&0\\
\end{pmatrix} \ \text{and}\ \begin{pmatrix}
 0& f_{J^{c}} \\
 f_{J}&0\\
\end{pmatrix}.
\]
By looking at the elements of smallest degree, we see that these ideals could not be equal unless $|I|=|J|=g+1$. Also, in case  $|I|=|J|=g+1$, the intersection $I\cap J$ is non-empty since $J \not= I^c$ and for $i \in I\cap J$ we recover $f_I$ as the smallest degree generator of $(f_i) \cap (f_I,f_{I^c})$.  
 
 There are $2^{2g+2}/4$ unordered pairs $\{I,I^c\}$  of even subsets of $\all$. Thus we get all $2^{2g}$ different two-torsion line bundles $\sL_I$ for even $I$. A similar argument applies to roots of $\sH$.
\end{proof}

\section{BGG for complete intersections of quadrics}\label{BGG}

This section provides what we need of the theory of \cite{BEH} and  \cite{Kapranov}. Let $P_X := k[V^{*}]/(q_{1},\ldots ,q_{c})$ be the homogeneous coordinate ring of the complete intersection $X = Q_{1}\cap \cdots \cap Q_{c} \subset \PP(V^{*})=\PP^{r-1}$ of $c$~quadrics $Q_{i} = V(q_{i})$ and choose a basis $x_1,\ldots,x_r$ of  $V^*$.  Write $B_{\ell}$ for the symmetric matrix with
$i,j$ entry
\[
b_{\ell, i,j} = \frac{1}{2}(q_{\ell}(x_{i}+x_{j})-q_{\ell}(x_{i})-q_{\ell}(x_{j})).
\]

Let $T=k[t_{1},\ldots,t_{c}]$ denote a polynomial ring in $c$ variables  each of degree 2 and let 
\[
q:\left\{\begin{array}{rcl}
T \tensor V &\lra &T\\
1\tensor v & \longmapsto & t_{1}q_{1}(v)+\cdots +t_{c}q_{c}(v)\end{array}\right.
\] 
denote the corresponding family of quadratic forms over  $\Spec T$.
Let 
\[C := \bigslant{\left(T \tensor \left(\bigoplus_{n} V^{\tensor n}\right)\right)}{(v\tensor v -q(v)\mid v \in V)}
\]
denote the  $\ZZ$-graded Clifford algebra of $q$, so that
$C$ is the quadratic dual of $P_X$ in the sense of \cite{PolishchukPositselski}.
The algebra $C$ is free as a $T$-module with basis
\[
e_{I}=e_{i_{1}}e_{i_{2}}\cdots e_{i_{k}}
\]
where $e_{1},\ldots,e_{r}$ is a basis of $V$ dual to $x_1,\ldots ,x_r$ and $I = \{i_{1}<i_{2}< \cdots < i_k \} \subset \{1,\ldots,r\}$ an ordered subset.
See for example \cite[Section~4.8]{Jacobson}.

\begin{theorem} Let $P_X$ be the homogeneous coordinate ring of a complete intersection of $c$ quadrics, and let $C$ denote the corresponding $\ZZ$-graded    
 Clifford algebra. Then $P_X$ and $C$ are a pair of Koszul dual graded algebras.
In particular
\[\Ext_{P_X}(k,k) \cong C \ \text{and}\ \Ext_C(k,k) \cong P_X.\]
\end{theorem}

\begin{proof}
See \cite{Sjodin}, \cite[Section 1]{Kapranov} and~\cite{PolishchukPositselski}.
\end{proof}

\begin{corollary} For any graded $P_X$-module $M$ the module
$\Ext_{P_X}(M,k)$ is a graded $C=\Ext_{P_X}(k,k)$-module.
\end{corollary}

The main result of this section is that for a graded $P_X$-modules $M$ with a linear resolution one can recover $M$ from the graded $C$-module $\Ext_{P_X}(M,k)$.

If $M$ is a (left) $P_X$-module and $N$ is a right $C$-module then we define an endomorphism of left \mbox{$P_X\otimes C$-modules}
\[
d: \Hom_{k}(N,M)\lra \Hom_{k}(N,M)
\]
taking $\phi \in \Hom_{k}(N,M)$ to $\psi$, where $\psi(n) = \sum_{i}x_{i}\phi(ne_{i})$.

Note that 
\begin{align*}
 d^{2}(\phi)(n) &= \sum_{i,j}x_{i}x_{j}\phi(n e_{i}e_{j})= \sum_{i\leq j}x_{i}x_{j}\phi(n (e_{i}e_{j}+e_{j}e_{i}))= \sum_{i\leq j}x_{i}x_{j}\phi(n \sum_{\ell}(t_{\ell}b_{\ell,i,j}))\\
&= \sum_{\ell}\sum_{i\leq j}b_{\ell,i,j}x_{i}x_{j}\phi(n t_{\ell})= \sum_{\ell}q_{\ell}(x)\phi(n t_{\ell})= 0.
\end{align*}

Thus, when $N$ is $\ZZ$-graded,  $\Hom_{k}(N,M)$ may be regarded as a complex of $P_X$-modules 
\[
\Hom_{k}(N,M)\,\colon\, \cdots \lra \Hom_{k}(N_{i},M) \lra \Hom_{k}(N_{i-1},M) \lra\cdots.
\]
When $M$ is $\ZZ$-graded and $N$ is  a $C-C$-bimodule, then $\Hom_{k}(N,M)$ may also be regarded as a complex of  right $C$-modules
\[
\Hom_{k}(N,M)\,\colon\, \cdots \lra \Hom_{k}(N, M_{i}) \lra \Hom_{k}(N,M_{i+1})\lra\cdots.
\]
Similar statements hold for $\Hom_{k}(M,N)$.

\begin{theorem}\label{linearResolutions}
 If the graded $P_X$-module $M$ has a linear free resolution, then the resolution may be written in the form
\[
\Hom_{k}(\Ext_{P_X}(M,k), P_X)
\]
 where we view $\Ext_{P_X}(M,k)$ as a graded $C=\Ext_{P_X}(k,k)$ module, and apply the construction above.
 \end{theorem} 

\begin{example} The complex $\Hom_{k}(C,P_{X})$,
\[
0 \lTo C_{0}^{*}\otimes_{k} P_{X}\lTo C_{1}^{*}\otimes_{k} P_{X} \lTo C_{2}^{*}\otimes_{k} P_{X} \lTo \cdots
\]
is isomorphic to the $P_{X}$-free resolution of $k$.
\end{example}

Note that this statement may be deduced from \cite[Corollary~3.2(iiM)]{PolishchukPositselski}. Since this result plays a crucial role in the proof of Proposition~\ref{bundle from ulrich},  we give a proof below. For our proof we need an explicit description  of the action of $\Ext^1_{P_X}(k,k)$ on $\Ext_{P_X}(M,k)$.
 
 To avoid keeping track of grading shifts we formulate this in case of a finitely generated module $M$ over a 
 Noetherian local ring $R$ with maximal ideal $\gm$. Let $(x_1,\ldots,x_r)$ denote minimal generators of $\gm$, and let
$e_{i}\in \Ext^{1}_{R}(k,k)$ be the extension
\[
e_{i}\, \colon\ 0\lra k\rTo^{x_{i}} E_{i}\rTo k \lra 0,
\]
where $E_{i} = R/(x_{1},\ldots,x_{i-1},x_{i}^{2}, x_{i+1},\ldots x_{r})$.
Let  
\[
\FF\, \colon\ \cdots \rTo^{d} F_{j}\rTo^{d} \cdots \rTo^{d} F_{0}
\]
be the minimal free resolution of a finitely generated  $R$-module $M$. Since the resolution $F$ is minimal the differential $d(f)$ of an element $f \in F_{j+1}$ can be written in the form 
$d(f)= \sum_{i=1}^r x_i f_i$ for $f_i \in F_j$.

\begin{lemma}\label{ext1action} Let  $\alpha \in \Ext_R^{j}(M,k)$ be a class represented by a map 
$\alpha': F_{j}\to k$. The element $\alpha e_{i} \in \Ext_R^{j+1}(M,k)$ is then represented by the map $\beta_{i}$ with
$\beta_{i}(f)=\alpha'(f_i)$ for $f \in F_{j+1}$ with differential $d(f)=\sum_{i=1}^rx_i f_i$.
\end{lemma}

\begin{proof}
We compute the image of $\alpha$ under the connecting homomorphism $\delta$
\[ \Ext^j_R(M,E_i) \rTo \Ext^j_R(M,k) \rTo^{\delta} \Ext^{j+1}_R(M,k) \rTo \Ext^{j+1}_R(M,E_i)\]
associated to the sequence $e_i$ above.
Consider the diagram
\[
\begin{diagram}
 F_{j+1}&\rTo^{d}& F_{j}\\
 \dTo^{\beta_{i}} & &\dTo^{\alpha''}&\rdTo^{\alpha'}\\
 k&\rTo^{x_{i}}& E_{i}&\rTo& k
\end{diagram}
\]
where $\alpha''$ is a lift of $\alpha'$ to $E_{i}$. The composition $\alpha' \circ d$ is zero since $\alpha'(\gm F_j)=0$.
Thus $\alpha'' \circ d$ factors over the map
\[
\beta_{i}\colon\left\{\begin{array}{rcl}
 F_{j+1}& \lra & k\\ f &\longmapsto & \alpha'(f_i).
\end{array}\right.
\]
This map is well-defined, \emph{i.e.} independent of the choice of $f_i$. Indeed, if $d(f)=\sum_{i=1}^r x_i f'_i$ is a different choice
for the presentation of $d(f)$ then $x_i(f_i-f'_i) \in (x_{1},\ldots,x_{i-1}, x_{i+1},\ldots x_{r})F_j$ which maps to zero in $E_i$.
\end{proof}

\begin{proof}[Proof of Theorem \ref{linearResolutions}]
Let
\[
\FF\,\colon\ \cdots \rTo^{d} F_{j}\rTo^{d} \cdots \rTo^{d} F_{0}
\]
be the minimal graded free resolution of $M$ as a $P_X$-module. Then
\[
F_j \cong \overline F_j \tensor_k P_X
\]
where $\overline F_j \cong F_j/\gm F_j$.
If $M$ has a linear resolution then we claim that the isomorphisms
\[
\left\{
\begin{array}{ccc}
 F_{j}=\overline F_j \tensor P_X &  \xrightarrow{\ \cong\ } & \Hom_k(\Hom_{P_X}(F_j,k),P_X)\\
 \overline f \tensor p &  \longmapsto &  \{\varphi\colon \alpha \mapsto \alpha(f)\tensor p\}
\end{array}\right.
\]
induce an isomorphism of complexes, \emph{i.e.} we have to show that these maps commute with differentials of the two complexes. Let $b_1,\ldots,b_\ell$ be a $k$-basis of $\overline F_j$ with dual basis $b_1^*,\ldots,b_\ell^*$ of $\overline F_j^*=\Hom_k(\overline F_j,k)=\Hom_{P_X}(F_j,k)$.

Consider an element $f=\overline f \otimes 1 \in F_{j+1}$. Then
\[
d(f)= \sum_{i=1}^r\sum_{\nu=1}^\ell c_{i\nu} b_\nu \tensor x_i \hbox{ with } c_{i\nu} \in k
\]
and we can take $f_i=\sum_{\nu=1}^\ell c_{i\nu} b_\nu$ for the coefficient of $x_{i}$ as in Lemma \ref{ext1action}.
The map 
\[
\{\varphi\colon \alpha \mapsto \alpha(f)\} \in \Hom_k\Big(\Hom_{P_X}(F_{j+1},k),P_X\Big)
\]
maps to
\[
\{\alpha' \mapsto \sum_{i=1}^r x_i \varphi(\alpha'e_i) \}  \in \Hom_k\Big(\Hom_{P_X}(F_{j},k),P_X\Big)
\]
by the definition of the differential  of $\Hom_k(\Ext_{P_X}(M,k),P_X)$.
We have
 \begin{align*}
 \sum_{i=1}^r x_i \varphi(\alpha'e_i) &=  \sum_{i=1}^r x_i \alpha'(f_{i}) \qquad \text{(by Lemma \ref{ext1action})} \\
 & = \sum_{i=1}^r x_i \alpha'(\sum_{\nu=1}^\ell c_{i\nu} b_\nu)
 \end{align*}
 In particular, for $\alpha'=b_\mu^*$ we obtain
 $b_\mu^* \mapsto \sum_{i=1}^r c_{i\mu} x_i$.  
 These values coincide with the values of the image of 
\[
d(f)=\sum_{i=1}^r\sum_{\nu=1}^\ell c_{i\nu} b_\nu \tensor x_i
\]
in $\Hom_k\Big(\Hom_{P_X}(F_j,k),P_X\Big)$, since $b_\mu^*(\sum_{\nu=1}^\ell c_{i\nu} b_\nu)= c_{i\mu}$.
\end{proof}

\begin{corollary}\label{linRes} Let $N$ be a graded left $C$-module.
The complex $\Hom_k(N,P_X)$ is a resolution if and only if $N \cong \Ext_{P_X}(M,k)$ up to shift where $M$ is a $P_X$-module with a linear resolution.
\end{corollary}
\begin{proof}
If $N \cong \Ext_{P_X}(M,k)$ up to shift where $M$ is a $P_X$-module with a linear resolution then by Theorem 2.4
the resolution of $M$ is $\Hom_k(N,P_X)$. Conversely, if the  complex $\Hom_k(N,P_X)$ is a resolution, then since it is linear we may take the module it resolves to be $M$.
\end{proof}

\section{Pencils of quadrics and hyperelliptic curves}\label{pencil section}

We now specialize to the case of a smooth intersection of two quadrics in $\PP^{2g+1}$
with coordinate ring $P_X = k[x_{1}, \dots, x_{2g+2}]/(q_{1}, q_{2})$. 
To simplify notation we write $s,t$ instead of $t_1,t_2$. Let $q=q(s,t) = sq_{1}+tq_{2}$ and let
$C = \Cliff(q)$ denote the $\ZZ$-graded Clifford algebra of $q$, so that $T = k[s,t]\subset C$.

As in Reid's thesis \cite{ReidThesis} we note that none of the quadrics in the pencil can have corank 2: for, if one of the quadrics had singular locus $L$ of dimension at least 2, then $X$ would be singular
at $L\cap X$. Further, by
Bertini's Theorem the general linear combination of the two quadrics is non-singular outside the intersection. But if it were
singular at a point of the intersection, then the intersection would be singular there too. Thus we  may assume that
one of the quadrics has full rank, and it follows that the two quadrics can be simultaneously diagonalized (see \cite[XII, Paragraph 6, Theorem 7]{Gantmacher}). Thus we may assume that the bilinear form 
 $q(s,t) = sq_{1}+tq_{2}$ is given by a diagonal matrix 
\[
\begin{pmatrix}
f_1  & &   0\\
& \ddots &  \\  
0 &  & f_{2g+2}  & \\
\end{pmatrix}
\]
with entries that are pairwise coprime linear polynomials $f_i \in k[s,t]$.
As in Section~\ref{sec:vector bundles} we denote by $f=\prod f_i$, and use the notation $f_I= \prod_{i\in I} f_i$.

We write
\[
C=C^{\ev} \oplus C^{\odd}
\]
for the decomposition of the Clifford algebra into its even and odd parts. As a $T = k[s,t]$-module,  $C$ is free with basis $e_I$ and 
\begin{equation}\label{monomialProducts}
e_{I}e_{J} =\epsilon(I,J)  f_{I\cap J}\; e_{I\Delta J}.
\end{equation}
with  the sign $\epsilon(I,J) = (-1)^{\sum_{i \in I} |\{ j \in J | j<i \}|}$.

Since
\[
\sum_{i \in I} \left|\{ j \in \all \mid j<i \}\right|=\sum_{i \in I} (i-1)
\]
and 
\[
\sum_{j \in \all} |\{i \in I | i<j \}| =\sum_{i \in I} (2g+2-i) \equiv \sum_{i\in I} (i-1)  \mod 2
\]
for even $I$, we see that $e_\all$ lies in the center  of the even Clifford algebra. 
Because
\[
\sum_{i=1}^{2g+2} (i-1) =\binom{2g+2}{2} \equiv {g+1} \mod 2,
\]
the element $e_\all$ satisfies the equation
\[
e_\all^2=(-1)^{g+1}f.
\]
To adjust for the sign we take $y=(\sqrt{-1})^{g+1}e_\all$ as a generator of the center of the even Clifford algebra over $k[s,t]$ so that $y^2=f$.
Note that the formula above  for the central element $y$ is only correct in the case of diagonal quadrics; for the general case see
\cite[Satz 1]{Haag}. 

Furthermore, for any $I$,
\[
e_I e_\all = (-1)^{\sum_{i \in I} (i-1)} f_I \; e_{I^c}\quad \text{and} \quad e_{I^c} e_\all = (-1)^{\sum_{i \in I^c} (i-1)} f_{I^c} \; e_{I}.
\]
Note that the signs in the two formulas differ by $(-1)^{g+1}$.
Thus with $R_{E}=k[s,t,y]/(y^2-f)$ the coordinate ring of the corresponding hyperelliptic curve,
the $R_E$-submodule of $C$ generated by $e_I$ and $e_{I^c}$ coincides with $H^{0}_{*}(\sL_I)$ from Definition \ref{LI}. 

Notice however, that here, differently than in section 1, $R_{E}$ is $2\ZZ$-graded as a subring of $C^{\ev}$. Thus, since the elements of $C^{\odd}$ have odd degree, we have to twist by an odd number
to obtain a non-trivial sheaf of $\sO_{E}$-modules.

We define $\sC^{\ev}=\widetilde{C^{\ev}}$ and $\sC^{\odd}=\widetilde{C^{\odd}(1)}$.
Hence multiplication in $C$ gives a map
\[
C^{\odd}(1)\times C^{\odd}(1) \lra C^{\ev}(2)
\]
which sheafifies  to a map
\[
\sC^{\odd}\otimes_{\sO_{E}} \sC^{\odd} \lra \sC^{\ev} \otimes \pi^{*} \sO_{\PP^{1}}(1)=\sC^{\ev} \otimes \sH.
\]

In summary, we get the following statement.

\begin{proposition}\label{evenOddCliff} Let $y=(\sqrt{-1})^{g+1}e_{\{1,\ldots,2g+2\}}$. The element $y$ is in the center of $C^{\ev}$ and satisfies the equation
$y^{2} = f$, where $f=\prod_{i=1}^{2g+2} f_i.$ If we write $R_{E} = k[s,t,y]/(y^{2} - f)$ then
the even Clifford algebra decomposes as an $R_E$-module as
\[
C^{\ev} = \bigoplus_{\substack{\{I,I^c\} \\ |I|\ \text{even}}} H^{0}_{*}(\sL_I).
\]
The odd part of the Clifford algebra decomposes as a right $R_E$-module as
\[
C^{\odd}(1) = \bigoplus_{\substack{\{I,I^c\}  \\ |I|\ \text{odd}}} H^{0}_{*}(\sL_I).
\]
Moreover, 
\[
\sC^{\odd} \cong  \sO_E(p) \otimes \sC^{\ev}
\]
where $p$ is any ramification point of $\pi\colon E \to \PP^1$.
\end{proposition}

\begin{proof}
This follows from Theorem~\ref{even and odd}. Note that  since  $p$ is a ramification point we have $\sO_{E}(2p) \cong \sH$
and the multiplication map $\sC^{\odd}\otimes \sC^{\odd} \to \sC^{\ev}\otimes \sH$ is compatible with the map
\[\sO_{E}(p) \tensor \sO_{E}(p) \lra \sO_{E}(2p) \cong \pi^{*}\sO_{\PP^{1}}(1)= \sH.\qedhere\]
\end{proof}

\begin{remark}\label{anticommuting} Notice  that $y$ and the elements of $C^{\odd}$ anti-commute by equation (\ref{monomialProducts}) applied to the case when $I$ is a singleton and
$J = \{1,\dots, 2g+2\}$.
\end{remark}

 The following result was proven in \cite{ReidThesis}.

\begin{lemma}\label{isotropic existence} Let $q_{1}, q_{2}$ be two quadratic forms on a $2g+2$-dimensional
vector space $V$ over $k$. The set of $g$-dimensional common isotropic subspaces of $q_{1}, q_{2}$
is non-empty and has dimension $\geq g$ locally at every point. 
\end{lemma}

\begin{proof}
Let $\sU$ be the universal sub-bundle on the Grassmannian $\GG := \GG(g,V)$. The forms $q_{i}$ define
homomorphisms $\Sym^{2}V^*\otimes_{k}\sO_{\GG} \to \sO_{\GG}$, and thus, by restriction,
sections of $\Sym^{2}(\sU^{*})$. The set of $g$-dimensional common isotropic subspaces is the common zero locus of these two sections. Computing the Chern class we see that the locus 
is non empty and, since
\[
\dim \GG(g,V) - 2\rank\ \Sym^{2}(\sU^{*})  = g(g+2) - 2\binom{g+1}{2} = g,
\]
the inequality on dimensions follows.
\end{proof}

We  return to the situation at the beginning of Section 3, with 
\[
P = k[x_{1},\ldots,x_{2g+2}] = \Sym(V^{*}).
\]
Let $U\subset V$ be a $g$-dimensional isotropic linear subspace and denote by
$P_{U} = \Sym(U^{*}) = P/(U^{\perp})$ its coordinate ring,  
where $U^{\perp}\subset V^*$ is the space of linear equations of the isotropic space $U$.

\begin{proposition}\label{linear PU}
Let $G = ks\oplus kt \cong k^2$
be the space of parameters for the family of quadratic forms $sq_1+tq_2$.
Considered as a $P_X$-module, $P_U$ has a linear free resolution.
 Moreover
\begin{equation}\label{ev}
\Ext^{2p}_{P_X}(P_U,k) = \bigoplus_{i} \left(\Lambda^{2i}U^{\perp} \tensor_k (\Sym_{p-i} G)^*\right)^*
\end{equation}
and 
\begin{equation}\label{odd}
\Ext^{2p+1}_{P_X}(P_U,k) = \bigoplus_{i}  \left(\Lambda^{2i+1}U^{\perp} \tensor_k (\Sym_{p-i} G)^* \right)^*.
\end{equation}
\end{proposition}

\begin{proof}
The ideal $(U^{\perp})$  contains the 2-dimensional vector space  $G:= \langle q_{1},q_{2}\rangle$. This ideal
is generated by  a regular sequence of linear forms, and the $P$-free resolution of $P_U = P/(U^{\perp})$ 
is thus a  Koszul complex
with underlying free module $P\otimes \Lambda U^{\perp}$. Let $\gamma: G \to P_{1}\otimes U^\perp$
be a map of vector spaces such that the composition  of $\gamma$ with the multiplication map
\[
G \lra P_{1}\otimes_{k} U^\perp \lra P_{2}
\]
is the inclusion of $G$ in $P_{2}$.

By \cite[Theorem 4]{Tate}, the minimal $P_X$-free resolution of $P_U$
 is the differential graded $R$-algebra
\[
P_X\otimes_{k} \Lambda U^\perp \otimes \left(\Sym G^{*}\right)^{*}.
\]
Here $U^\perp$ has internal degree 1 and homological degree 1, while $G^*$ has internal degree 2 and homological degree 2, and the component of the differential $G = \left(\Sym^1 G^{*}\right)^{*}\to P_{X}\otimes_{k}U^\perp$ is induced by $\gamma$.
 
 This resolution is linear, and has
 degree $j$ term
\[
P_{X}\otimes_{k}\left(\bigoplus_{j=a+2b} (\Lambda^{a} U^{\perp})\tensor_k (\Sym_bG)^*\right).\qedhere
\]
\end{proof}

Let $T=\Sym G\cong k[s,t]$  and write
\[
F_{U}=\Ext^{\ev}_{P_X}(P_U,k) = \bigoplus_i \left(\left(\Lambda^{2i} U^{\perp}\right)^*\tensor_k T(-i)\right)
\]
regarded as a module over $\Ext_{P_{X}}^{\ev}(k,k) =  C^{\ev}$ \emph{via} the Yoneda pairing.

\begin{proposition}\label{degreeOfFu} The sheafification $\sF_{U}$ of $F_U$ as an $\sO_E$-module is a vector bundle of $\rank \sF_{U}= 2^{g}$ and degree
$\deg \sF_{U}= g2^{g-1}$ on $E$. Moreover,
\[
H^0_*(\sF_U) = Ext^{\ev}(P_U, k),
\]
and
\[
 H^0_*(\sF_U(p)) = Ext^{\odd}(P_U, k).
\]
\end{proposition}

\begin{proof}
It follows from the formulas above that the sheafification $\sF_{U}$ of $F_U$ as an $\sO_E$-module is a vector bundle of rank equal to
$\dim_k(\Lambda^{\ev}U^{\perp}) /2 = 2^{g}$. Moreover
\[
\deg \pi_{*} \sF_{U} = -\sum_{i\ge 0} i\binom{g+2}{2i} = -(g+2)2^{g-1}.
\]
By  Proposition \ref{vectorBundlesOnE},  $\sF_{U}$ has degree
\[
\deg \sF_{U}=(g+1)\rank \sF_{U}+\deg \pi_{*} \sF_{U}=(g+1)2^{g}-(g+2)2^{g-1}=g2^{g-1}.
\]

The first displayed formula is immediate from the definition of $F_U$, while the second follows
from the equality $\sC^{\odd} = \sC^{\ev}(p)$.
\end{proof}

\begin{theorem} The endomorphism bundle of $\sF_{U}$
is isomorphic as an $\sO_E$-algebra to the sheafified  even Clifford algebra $\sC^{\ev}$; that is,
\[
\sEnd_{E}(\sF_{U}) \cong \sC^{\ev}.
\]
\end{theorem}

\begin{proof}
Let $(a,b)\in \PP^{1}$ be a point that is not a branch point of $\pi$. 
The algebra $\pi_{*} \sC^{\ev}$ is a sheaf of algebras whose fiber at  $(a,b)$ is isomorphic to the product
of the fibers of $\sC^{\ev}$ at the two preimages of $(a,b)$ in $E$. On the other hand, the fiber of $\pi_{*} \sC^{\ev}$
is the even Clifford algebra
of the nonsingular quadratic form $a{q_{1}}+bq_{2}$. Thus it is a semisimple algebra with 2-dimensional center generated
over $k$ by $y$. Since we have assumed that $k$ is algebraically closed, this center is $k\times k$. The corresponding decomposition of the push forward of $\sC^{\ev}$ as a direct product is the unique decomposition as the product of two algebras. Thus
the fibers of $\sC^{\ev}$ at points of $E$ other than the ramification points are simple algebras by \cite[Theorem 4.13]{Jacobson}.

Since  $F_U$ is an $R_E-C^{\ev}$ bimodule we have  an $\sO_E$-algebra homomorphism 
\[
\phi: \sC^{\ev} \lra \sEnd_{E}(\sF_{U}).
\]
Since the general fiber of $\sC^{\ev}$ is simple, the kernel of this homomorphism must be torsion, and thus 0.
The source and target of $\phi$ are vector bundles of the same rank. By Proposition~\ref{evenOddCliff} the sheaf $\sC^{\ev}$ is a
sum of the degree 0 line bundles $\sL_I$, and since the endomorphism bundle also has degree 0, the map is
an~isomorphism.
\end{proof}

\begin{corollary}[Morita equivalence, see \protect{\cite[Chapter 2]{Bass}}]\label{morita} 

The $\sO_E- \sC^{\ev}$ bimodule $\sF_U$ defines an equivalence of module categories
\[
\left\{\begin{array}{ccc}
\sO_E -mod  &\longleftrightarrow &mod -\sC^{\ev}\\
 \sL  & \longmapsto &\sL \otimes_{\sO_E} \sF_U \\
\sG\otimes_{\sC^{\ev}} \sF_U^* & \longleftmapsto & \sG  \\
\end{array}\right.
\]
where $\sF_U^* = \sH om_{\sO_E}(\sF_U, \sO_E)$.
\end{corollary}

\begin{corollary}[Reid, 1972 \cite{ReidThesis}]\label{ReidsTheorem} 
Let $X=Q_1\cap Q_2 \subset \PP^{2g+1}$ be a smooth intersection of two quadrics and let
$E$ be the corresponding hyperelliptic curve. 
Let $U_{0}\subset V$ be a $g$-dimensional linear subspace such that $\PP(U_{0}^{*})\subset X$. Then the map
\[
\varphi\colon \left\{\begin{array}{ccl}
\{ U \in \GG(g,V) \mid \PP(U^{*}) \subset X \} & \lra & \Pic^0(E)\\
U &\longmapsto & \sF_U \otimes _{\sC^{\ev} }\sF_{U_{0}}^{*}
\end{array}\right.
\]
is a bijection. If the ground field $k$ has characteristic 0, it is an isomorphism.
\end{corollary}

\begin{proof}  
By Lemma \ref{isotropic existence},  a space $U_{0}$ of dimension $g$ such that $\PP(U_{0}^{*})\subset X$ exists.
We claim that $\sF_U \otimes _{\sC^{\ev} }\sF_{U_{0}}^{*}$ is an element of $\Pic^{0}(E)$.
We know by Corollary~\ref{morita} that $\sF_{U_{0}}$ and $\sF_U$  both define Morita equivalences.
Hence $\sL:= \sF_U \otimes _{\sC^{\ev} }\sF_{U_{0}}^{*}$ must be an invertible object in $\sO_E-mod$, hence a line bundle.
This line bundle  has degree 0 since $\sF_U \cong \sL \otimes \sF_{U_{0}}$ and both vector bundles have the same degree.

The map $\varphi$
 is injective because
we can recover $U$ from $\sF_{U}\cong \sL \otimes \sF_{U_{0}}$ as follows: 
by Proposition~\ref{bundle from ulrich} (3) below, we can recover the $C=\Ext_{P_X}(k,k)$-module $\Ext_{P_X}(P_{U},k)$ from $\sF_{U}$. 
The free resolution of $P_{U}$, hence $U^{\perp}$, can be obtained from $\Ext_{P_X}(P_{U},k)$ by Theorem \ref{linearResolutions}.

Since the source and target  of $\varphi$ are projective and the target is connected, smooth, and of the same
dimension as the source, the map is a surjective, hence a bijection. 
In case the ground field $k$ has characteristic 0 $\varphi$ is thus an isomorphism. If $k$ has positive characteristic it could be a purely inseperable morphism.
 Miles Reid proved in \cite[Theorem 4.8]{ReidThesis} that $\{ U  \in \GG(g,V) \mid \PP(U^{*}) \subset X \} $ and $\Pic^0(E)$ are isomorphic for arbitrary characteristic.
\end{proof}

\begin{remark} Our Macaulay2 package \cite{EKS} computes the action of $\Pic^0(E)$ on the space of maximal isotropic subspaces \[
\GG(g,X)=\left\{ U \in \GG(g,V) | \PP(U^{*}) \subset X \right\}.
\]
For a different approach to the group law on $\Pic^0(E)$ in terms of $\GG(g,X)$ see~\cite{Donagi}.
\end{remark}

\section{Tate resolutions of $P_X$-modules from Clifford modules}\label{sec:Tate}

The constructions in this section are inspired by the theory of Cohen--Macaulay approximations of Auslander and Buchweitz \cite{Auslander-Buchweitz} and the construction of Tate resolutions as in \cite{EisenbudSchreyer}.  Let $R$ be a Noetherian local or graded Gorenstein ring, and let $M$ be a finitely generated $R$-module. Let $F$ be the minimal $R$-free resolution of M:
\[
 0 \lTo M \lTo F_0 \lTo F_1 \lTo F_2 \lTo  \cdots .
\]

We will use the notation $N^*=\Hom_{R}(N,R)$ for the dual of an $R$-module $N$. If $N$ is a maximal Cohen--Macaulay (MCM) module, that is, an $R$-module of depth $\dim R$, then  we have  $(N^*)^* \cong N$, because $R$ is Gorenstein.

The Tate resolution associated to $M$ is a doubly infinite exact complex of free $R$-modules obtained as follows:
The \supth{i} syzygy module $M_i=\ker\left(F_{i-1} \to F_{i-2}\right)$ is an MCM module when $i>\dim R$,
so \mbox{$M_i^* = \ker\left(F_i^* \to F_{i+1}^*\right)$} is also an MCM module.

Choose an integer $i> \dim R$, and let 
\[
\cdots \rTo G_{i-2}\rTo G_{i-1} \rTo M_i^* \rTo 0
\]
be a minimal free resolution of $M_i^*$.
The Tate resolution $\TT(M)$ of $M$ is obtained by splicing the dual complex $G^*$ with the complex $F_i \lTo F_{i+1} \lTo \cdots $ to a doubly  infinite complex  
\[
\TT(M)\,\colon \  \cdots \lTo G_{i-2}^* \lTo G_{i-1}^* \lTo F_{i}  \lTo F_{i+1}  \lTo \cdots
\]
of free graded $R$-modules. This is an exact complex because both $M_i=\ker\left(F_{i-1} \to F_{i-2}\right)$ and \mbox{$M_i^*\cong \ker\left(F_i^* \to F_{i+1}^*\right)$} are MCM modules. Up to isomorphism this complex is independent of the choice of $i$ and the choice of the minimal free resolutions. The dual complex $\TT(M)^*$ is exact as well.

\begin{example}\label{matrix factorization example} In case  of a hypersurface ring $R=P/(f)$ the Tate resolutions are the double infinite periodic complexes 
\[
\cdots \lTo^{\overline \phi} R^{n} \lTo^{\overline \psi}  R^{n}  \lTo^{\overline \phi}  R^{n} \lTo^{\overline \psi} \cdots
\]
obtained from matrix factorizations $(\phi,\psi)$ of $f$, \emph{cf.}~\cite{Eisenbud80}.
\end{example}

\begin{remark}
 Auslander and Buchweitz~\cite{Auslander-Buchweitz} used Tate resolutions to define the MCM approximation 
 of $M$ for arbitrary Cohen--Macaulay rings. When $R$ is Gorenstein, as in our case, we set $M^{\es} = \coker\left(G_1^* \to G_0^*\right)$, the \emph{essential MCM approximation}, so that $M^{\es}$ is an MCM over $R$. By \cite{Auslander-Buchweitz} there is an induced map $M^{\es}\to M$ and  the modules $M$ and $M^{\es}$ have free resolutions that differ
 in only finitely many terms: If $R^n\to M$ is a map from a graded free $P_X$ module such that 
\[
0 \lTo M \lTo M^{\es} \oplus R^n 
\]
is a surjection, then the kernel of this homomorphism has a finite free resolution of length ${\rm codepth}\; M -1$. 
Auslander--Buchweitz define this homomorphism to be the MCM approximation of $M$ 
if $n$ is taken to be~minimal. 
\end{remark}

\begin{proposition}\label{Tate of linear CM} Let $P_{X}=P/(q_{1},\ldots,q_{c})$ be the homogeneous coordinate ring of a complete intersection of quadrics. Let $M$ be a $P_{X}$-module which has a linear resolution as a $P$-module. Then $\Ext_{P_{X}}(M,k)$ is a $C=\Ext_{P_{X}}(k,k)$-module which is free as a $k[t_{1},\ldots,t_{c}]$-module. If moreover $M$ is a Cohen--Macaulay $P_{X}$-module of codimension $\ell$ then the Tate resolution of $M$ has the form
\[
\begin{diagram}
 \cdots & \lTo &P_X^{b_{-2}}(3)&\lTo& P_X^{b_{-1}}(2)&\lTo &P_X^{b_{0}}(1)& \lTo &\cdots&\lTo & P_{X}^{b_{\ell}}(-\ell+2)&\lTo &0\\
 &&&&\uTo^{\phi_{0}}&&\uTo^{\phi_{1}}&&&&\uTo^{\phi_\ell}\\
& &0 &\lTo&P_X^{a_0}&\lTo&P_X^{a_1}(-1)&\lTo&\cdots&\lTo&P_X^{a_\ell}(-\ell)&\lTo&\cdots
\end{diagram}
\]
with $b_{\ell-i }=a_i$
with an overlap of length $\ell$. The bottom row, which is a quotient complex, is the Eisenbud--Shamash resolution of $M$ as a $P_{X}$-module, and the top row, a subcomplex, is its $P_{X}$ dual.
\end{proposition}

\begin{proof}
As in the special case explained in the proof of Proposition~\ref{linear PU},  the Eisenbud--Shamash graded free resolution of $M$ as a $P_X$ module
 \cite[Theorem 7.2]{Eisenbud80} can be constructed from a series of
higher homotopies on a graded $P$-free resolution $F$ of $M$. Because the $q_i$ have degree 2, all the higher homotopies are linear maps, so the construction yields a minimal linear resolution of $M$ whose underlying graded free module is a divided power algebra over $P_X$ on $c$ generators tensored with the underlying module of $F$, and this implies that $\Ext_{P_X}(M,k)$ is a free module over the dual algebra, $k[t_1,\ldots, t_c]$.

If $M$ is Cohen--Macaulay of codimension $\ell$ then the \supth{(\ell+1)} syzygy of $M$ is a maximal Cohen--Macaulay module,
and by \cite{EisenbudSchreyer} the Tate resolution of $M$ has the given form.
\end{proof}

In~\cite{EisenbudSchreyer} there is an explicit description of all maps in the Tate resolution in case of a nested pair of complete intersections such as the following.

\begin{example}\label{g=3 example} Consider the coordinate ring $P_U$ of a $g$-dimensional isotropic subspace $U$ in the complete intersection $X$ of two quadrics as a $P_X$-module. 
 The Tate resolution $\TT(P_U)$  has an overlap of length $\ell=\codim_X \PP(U^*)=2g-1-(g-1)=g$.
 In case $g=3$ it has Betti table
\[
\begin{matrix}
      \cdots&28&20&12&5&1&&&\\
&&&1&5&12&20&28&36&\cdots\\
\end{matrix}
\]
The vertical maps in the display of $\TT(P_U)$ are  northwest diagonal maps in the Betti table, which are represented by matrices of quadratic forms. For example the map $\phi_0$ as in Proposition \ref{Tate of linear CM} is given by a $20\times 1$ matrix of quadrics, represented in the Betti table by the northwest map from the left-most $1$ on the lower to the $20$ in the upper row. 
For arbitrary $g$ we obtain the formulas
\[
a_{2p}= \sum_{i=0}^p (p-i+1)\binom{g+2}{2i}\quad \text{and}\quad a_{2p+1} =  \sum_{i=0}^p (p-i+1)\binom{g+2}{2i+1}
\]
for the ranks $a_i$ in the lower row of the diagram above from the equations~\eqref{ev} and~\eqref{odd} in Section~\ref{pencil section}.
\end{example}

\begin{theorem}\label{construction of TN}
 Let $C=\Cliff(q_1,q_2)$ be the Clifford algebra over $k[s,t]$ of a nonsingular complete intersection of two quadrics in $\PP^{2g+1}$. Let $N$ be a graded $C$-module that is free as a $k[s,t]$-module, and such that the corresponding vector bundles $\sN^{\ev}=\widetilde{N^{\ev}}$ and  $\sN^{\odd}=\widetilde{N^{\odd}(1)}$  defined on the associated hyperelliptic curve $E$ satisfies
\[
 \sN^{\odd} \cong \sN^{\ev} \otimes_{\sC^{\ev}} \sC^{\odd}.
\]
Let $p\in E$ be a ramification point. There is a doubly infinite exact complex
\[ 
\TT(N)\,\colon\ \cdots \rTo F_i \rTo F_{i+1} \rTo \cdots
\]
of free modules $F_i = P_X^{a_i}(i)\oplus P_X^{b_i}(i+1)$
with Betti numbers $a_i = h^1(\sN^{\ev}(ip))$ and $b_i = h^0(\sN^{\ev}((i+1)p))$.
In terms of this decomposition, the complex $\TT(N)$ takes the form 
\[
\begin{matrix}
\to&H^1(\sN^{\ev}) \otimes_k P_X & \to & H^1(\sN^{\ev}(p))\otimes_k P_X(1) &\to& H^1(\sN^{\ev}(2p)) \otimes_k P_X(2)&\to   \cr
\searrow&\oplus  &\searrow& \oplus & \searrow & \oplus  & \searrow \cr
\to&H^0(\sN^{\ev}(p))\otimes_k P_X(p)& \to & H^0(\sN^{\ev}(2p)) \otimes_k P_X(2)  &\to & 
H^0(\sN^{\ev}(3p))\otimes_k P_X(3) &\to .
\end{matrix}
\]
\end{theorem}

\begin{proof} We will use the notations $x_i, e_i$ as defined in Section~\ref{BGG}. Consider the  sequence of maps
\[ 
\cdots \rTo^d N_{i-1} \otimes_k P   \rTo^d N_{i} \otimes_k P  \rTo^d N_{i+1} \otimes_k P \rTo^d \cdots
\]
defined by  $d(n\otimes_k r) = \sum_{i=1}^{2g+2} ne_i  \otimes_k x_ir $.

Computations similar to that at the beginning of Section~\ref{BGG} show that
\[
 d^{2}(n\otimes_k r) = \sum_{i,j} (n e_{i}e_{j}) \otimes_{k} (x_{i}x_{j} r ) 
= ns \otimes_k q_1(x)r + nt \otimes_k q_2(x)r
= n \otimes_{k[s,t]}  (sq_1(x)+tq_2(x))r,
\]
where the last step uses the identification $N \otimes_k P = N \otimes_{k[s,t]} P[s,t]$.

Set $A:=N^{\ev} \otimes_k P$ and $B:=N^{\odd} \otimes _k P$.
The map $d$ induces a matrix factorization 
\[ (A \lra B(0,1),\ B(0,1)\lra A(1,2))\]
of $sq_1+tq_2$ over the bi-graded polynomial ring  $k[s,t,x_1,\ldots,x_{2g+2}]$.
As in Example~\ref{matrix factorization example}, this matrix factorization induces a  $2$-periodic resolution
\[
 \cdots \rTo \overline B(-1,-1) \rTo \overline A  \rTo \overline B(0,1) \rTo \overline A(1,2) \rTo \cdots
\]
where $\overline A$ and $\overline B$ are restrictions of $A$ and $B$ to $k[s,t,x_1,\ldots,x_{2g+2}]/(sq_1+tq_2)$.

Sheafifying with respect to the variables $(s,t)$ we get a doubly infinite exact complex
\[
 \cdots \rTo \widetilde B(-1,-1) \rTo \widetilde A  \rTo \widetilde B(0,1) \rTo \widetilde A(1,2) \rTo \cdots
\]
of direct sums of line bundles on 
the hypersurface  $V(sq_1+tq_2) \subset \PP^1 \times \AA^{2g+2}$. 

We define an exact complex of $\sO_{\PP^1}\otimes P_X$-modules
by factoring out $q_1$ on the set $t\neq 0$ and $q_2$ on the set $s\neq 0$, identified on the set where neither $s$ nor $t$ is zero with
$k[s/t, t/s]\otimes P/(q_1, q_2)$.

 Since the central element $y$ of the even Clifford algebra anti-commutes with the action of the $e_i$ on $N$ by Remark~\ref{anticommuting} we may regard this also as a complex of $\sO_E \otimes P_X$-modules that are box products of locally free $\sO_E$-modules with graded free $P_X$-modules,
\[
\TT\,\colon\ \cdots \rTo \sA_E \boxtimes P_X \rTo  \sB_E \boxtimes P_X(1) \rTo \sA_E(1) \boxtimes P_X(2) \rTo \cdots,
\]
where use the fact that $\sO_{E}(1) \cong \sO_{E}(2p)$. Here $\sA_E = \sN^{\ev}$ and 
$
\sB_E 
$ 
is isomorphic to
\[
\sN^{\odd} = \sN^{\ev}\otimes_{\sC^{\ev}} \sC^{\odd} =\sN^{\ev}(p)
\]
by Proposition~\ref{evenOddCliff},
where the action of $y$ on $\sB_{E}$ is induced by the action of $-y$ on $N^{\odd}$. 
Thus these are the vector bundles on $E$ defined by the action  of $y$ or $-y$ on the even and odd part of $N$ respectively. 
In other words, $\sB_{E} \cong \iota^{*}\sN^{\odd}$, where $\iota\colon E\to E$  denotes the covering involution of $E \to \PP^{1}$.

\medskip
Let $\rho\colon E \times \Spec P_X \to \Spec P_X $ denote the second projection. The desired Tate resolution $\TT(N)$ associated to the Clifford module $N$ is essentially $R\rho_* \TT$. Since $\TT$ is a complex, we get a spectral sequence, which we analyze as follows:
truncate $\TT$ on the left to obtain a left bounded complex
\[
L_i \rTo \sA_E(i) \boxtimes P_X(2i) \rTo  \sB_E(i) \boxtimes P_X(2i+1) \rTo \sA_E(i+1) \boxtimes P_X(2i+2) \rTo \cdots,
\]
and take a \v Cech resolution on $E$ coming from a covering with two affine open subsets. We obtain a  double complex:
\begin{equation} \tag{$*_i$} \label{ref*}
\begin{gathered}
\xymatrix{
 0  &0 & 0  &  \cr
 C^1(L_{i}) \ar[u] & C^1(\sA_E(i)) \boxtimes P_X(2i) \ar[r] \ar[u] & C^1(\sB_E(i)) \boxtimes P_X(2i+1) \ar[r] \ar[u]& 
 \cdots \cr
 C^0(L_{i}) \ar[u] & C^0(\sA_E(i)) \boxtimes P_X(2i) \ar[r] \ar[u] & C^0(\sB_E(i)) \boxtimes P_X(2i+1) \ar[r] \ar[u]&
 \cdots \cr
  0 \ar[u] & 0 \ar[u]  & 0 \ar[u]&  \cr 
}
\end{gathered}
\end{equation}
The vertical homology of this double complex is a box product with the cohomology of $\sA_E$ and $\sB_E$ and  their twists.
The $E_2$-differentials of the  spectral sequence of the double complex can be lifted to maps of the form
$H^1(\sA_E) \otimes P_X \to H^0(\sA_E(1)) \otimes P_X(2)$ 
on the $E_1$-page of the sequence.
To do this, we choose $k$-vector space splittings $h$ of the  \v Cech sequence
\begin{equation*}\label{cech} \tag{$\alpha$}
0 \rTo H^0(\sA_E) \rTo C^0(\sA_E) \rTo C^1(\sA_E) \rTo H^1(\sA_E) \to 0
\end{equation*}
and the corresponding sequences $(\alpha_i)$ and $(\beta_i)$ for the sheaves $\sA_E(i)$'s and  $\sB_E(i)$'s respectively. We define the map 
\[
H^1(\sA_E) \otimes P_X \rTo H^0(\sA_E(1)) \otimes P_X(2)
\]
as the composition
\[
\xymatrix{ 
H^1(\sA_E) \otimes  P_X \ar[d]^{h \otimes \id} && \cr
C^1(\sA_E) \boxtimes P_X  \ar[r]  & C^1(\sB_E) \boxtimes P_X(1) \ar[d]^{h \boxtimes \id}& \cr
 &C^0(\sB_E) \boxtimes P_X(1) \ar[r]&C^0(\sA_E(1))\boxtimes P_X(2)\ar[d]^{h \otimes \id} \cr
 && H^0(\sA_E(1)) \otimes P_X(2) .
}
\]
Abusing notation we write $\tilde h$ for all south arrows,  $\tilde \partial$ for all north arrows, and $\varphi$ for all east arrows
in the corresponding diagram 
\begin{equation} \tag{{4}} \label{diag4}
\xymatrix{ 
H^1(\sA_E(i)) \otimes  P_X(2i) \ar[d]^{\tilde h} \ar[r]^\varphi &H^1(\sB_E(i)) \otimes  P_X(2i+1)\ar[d]^{\tilde h} \ar[r]^\varphi& H^1(\sA_E(i+1)) \otimes  P_X(2i+2) \ar[d]^{\tilde h}\cr
C^1(\sA_E(i)) \otimes  P_X(2i) \ar[d]^{\tilde h} \ar[u]^{\tilde \partial} \ar[r]^\varphi &C^1(\sB_E(i)) \otimes  P_X(2i+1) \ar[d]^{\tilde h} \ar[u]^{\tilde \partial} \ar[r]^\varphi \ar[r]& C^1(\sA_E(i+1)) \otimes  P_X(2i+2) \ar[d]^{\tilde h} \ar[u]^{\tilde \partial} \cr
C^0(\sA_E(i)) \otimes  P_X(2i) \ar[d]^{\tilde h} \ar[u]^{\tilde \partial} \ar[r]^\varphi &C^0(\sB_E(i)) \otimes  P_X(2i+1) \ar[d]^{\tilde h} \ar[u]^{\tilde \partial} \ar[r]^\varphi \ar[r]& C^0(\sA_E(i+1)) \otimes  P_X(2i+2) \ar[d]^{\tilde h} \ar[u]^{\tilde \partial} \cr
H^0(\sA_E(i)) \otimes  P_X(2i) \ar[u]^{\tilde \partial} \ar[r]^\varphi &H^1(\sB_E(i)) \otimes  P_X(2i+1)\ar[u]^{\tilde \partial}\ar[r]^\varphi& H^0(\sA_E(i+1)) \otimes  P_X(2i+2) \ar[u]^{\tilde \partial}\cr
}
\end{equation}
with four rows. 

For  $\alpha \in H^1(\sA_E) \boxtimes  P_X$ we have
\begin{align*}
&\alpha = \tilde \partial \tilde h \alpha & & \hbox{since } \partial h = \id_{H^1} \\
\Rightarrow \quad & \varphi \alpha= \tilde \partial \varphi \tilde h \alpha& &\hbox{since }[\varphi,\tilde \partial]=0 \\
\Rightarrow \quad & \tilde h\varphi \alpha= -\tilde \partial \tilde h \varphi \tilde h \alpha + \varphi \tilde h \alpha && \hbox{since } \partial h +h \partial =\id_{C^1}\\
\Rightarrow \quad &\varphi \tilde h\varphi \alpha= -\varphi \tilde \partial \tilde h \varphi \tilde h \alpha && \hbox{since }\varphi^2=0 \\
\Rightarrow \quad &\varphi \tilde h \varphi \alpha= - \tilde \partial \varphi \tilde h \varphi \tilde h \alpha && \hbox{since }[\varphi,\tilde \partial]=0\\
\Rightarrow \quad &\tilde h \varphi \tilde h \varphi \alpha=  \tilde \partial \tilde h \varphi \tilde h \varphi \tilde h \alpha - \varphi \tilde h \varphi \tilde h \alpha && \hbox{since }\partial h +h \partial =\id_{C^0}\\ 
\Rightarrow \quad &\varphi \tilde h \varphi \tilde h \varphi \alpha=  \tilde \partial \varphi \tilde h \varphi \tilde h \varphi \tilde h \alpha && \hbox{since }\varphi^2=0 \hbox{ and } [\varphi,\tilde \partial]=0\\
\Rightarrow \quad &\tilde h \varphi \tilde h \varphi \tilde h \varphi \alpha=   \varphi \tilde h \varphi \tilde h \varphi \tilde h \alpha && \hbox{since }h\partial=\id_{H^0}\\
\Rightarrow \quad &(\tilde h \varphi \tilde h \varphi \tilde h) \varphi =   \varphi (\tilde h \varphi \tilde h \varphi \tilde h).&&
\end{align*}
Thus with the lifted maps we obtain a double complex, whose total complex is our desired complex $\TT(N)$:
\[
\begin{matrix}
\to&H^1(\sA_E) \otimes P_X & \to & H^1(\sB_E)\otimes P_X(1) &\to& H^1(\sA_E(1)) \otimes P_X(2)   \cr
\searrow&\oplus  &\searrow& \oplus & \searrow & \oplus  & \cr
\to &H^0(\sB_E)\otimes P_X(1)& \to & H^0(\sA_E(1)) \otimes P_X(2)  &\to & 
H^0(\sB_E(1))\otimes P_X(3) .
\end{matrix}
\]
The right truncated complexes are exact except at the first two position since the spectral sequence of \eqref{ref*}
 converges to the cohomology of $L_i$. Since we can take $i$ arbitrarily large negative, the complex $\TT(N)$ is~exact.
\end{proof}

\begin{proposition}\label{bundle from ulrich}
Let $M$  be a $P_X$-module with a linear resolution as an $P$-module. Then 
\begin{enumerate}
  \item\label{item1} $N=\Ext_{P_X}(M,k)$ is a $C=\Ext_{P_X}(k,k)$-module which is free as an $k[s,t]$-module. 
 \item\label{item2} The sheafifications  $\sN^{\ev}$ and $\sN^{\odd}=\widetilde{N^{\odd}(1)}$ satisfies 
\[
 \sN^{\odd} \cong \sN^{\ev} \otimes_{\sC^{\ev}} \sC^{\odd}.
\]
\item\label{item3} $N = H^0_*(\sN^{\ev})\oplus H^0_*(\sN^{\odd})(-1)$ and the $C$-module $N$ is determined by the $\sC^{\ev}$-module $\sN^{\ev}$.
\item\label{item4} The $P_X$-dual complex
$\TT(N)^*$ is the Tate resolution $\TT(M)$ of $M$. 
\end{enumerate}
\end{proposition}

\begin{proof} \eqref{item1} Let 
$
0 \to F_{c} \to \cdots \to F_{1} \to F_{0} \to M \to 0
$
be the linear $P$-resolution of $M$. Then by the Eisenbud--Shamash construction \cite[Theorem 7.2]{Eisenbud80},
$\Ext_{P_{X}}(M,k)=N=N^{\ev}\oplus N^{\odd}$ is a free $k[s,t]$-module.

\eqref{item2} We have
\[ 
\rank_{k[s,t]} N^{\ev}= \sum_{i\ge 0}  \rank_{P} F_{2i}\quad \text{and}\quad \rank_{k[s,t]} N^{\odd} = \sum_{i \ge 1}  \rank_{P} F_{2i+1}.
\]
Since $\sum_{i=0}^{c} (-1)^{i} \rank_P F_{i}=0$ the $k[s,t]$-modules $N^{\ev}$ and $N^{\odd}$ have equal rank.
Theorem~\ref{linearResolutions}  shows that the minimal free $P_{X}$-resolution of $M$ is isomorphic to
 $\Hom_{k}(\Ext_{P_{X}}(M,k),P_{X})$. From this construction we see that
if one of the maps
\[
\Ext_{P_{X}}^{i}(M,k) \times \Ext^{1}_{P_{X}}(k,k) \rTo \Ext^{i+1}_{P_{X}}(M,k)
\]
were not surjective, then there would be a generator of the module
$\Hom_{k}(\Ext^{i+1}_{P_{X}}(M,k),k)$ which maps to zero in the complex. This is not possible because the complex is minimal.
We conclude that the map
\[
\cN^{\ev}\otimes_{\sC^{\ev}} \sC^{\odd} \rTo \sN^{\odd}
\]
is a surjective morphism of $\sO_{E}$-vector bundles of the same rank and  hence an isomorphism
of $\sC^{\ev}$ modules. 

\eqref{item3} It follows that
\[
\cN^{\odd}\otimes_{\sC^{\ev}} \sC^{\odd} \cong \cN^{\ev}\otimes_{\sC^{\ev}} \sC^{\odd}\otimes_{\sC^{\ev}} \sC^{\odd} \rTo \sN^{\ev}
\otimes \sH
\]
is also an isomorphism. 

The formula for $N$ follows because $N$ is a free $k[s,t]$-module. Since $C^{\ev} = H^0_*(\sC^{\ev})$
and  \mbox{$C^{\odd} = H^0_*(\sC^{\odd})(-1)$} the maps above determine the maps 
$N^{\ev}\otimes_kC^{\odd} \to N^{\odd}$
and
$N^{\odd}\otimes_kC^{\odd} \to N^{\ev}$,
and thus the $C$-module structure on $N$.

\eqref{item4} By parts~\eqref{item1} and~\eqref{item2} we can apply Theorem~\ref{construction of TN}. The dual of the $H^{0}$-strand of $\TT(N)$ coincides with $\Hom_{k}(\Ext_{P_{X}}(M,k),P_{X})$ by construction.
Since $\TT(N)^{*}$ and $\TT(M)$ are exact minimal complexes which coincide for large homological degree, they are isomorphic.
\end{proof}

\begin{example}\label{g=3example bis} 
 Thus in case $g=3$ the Betti table 
\[
\begin{matrix}
      \cdots&28&20&12&5&1&&&\\
&&&1&5&12&20&28&36&\cdots\\
\end{matrix}
\]
of the Tate resolution of $M=\TT(H^0_*(\sF_U\otimes_{\sC^{\ev}}\sC))$  
has a second interpretation.
It is also the cohomology table
\[
\left(h^i(\sF_U((j+1-i)p)\right)_{\substack{i=0,1\\ j\in \ZZ}}
\]
of $\sF_U$ as a vector bundle on the hyperelliptic curve $E$. 
\end{example}

\begin{theorem}\label{UlrichCondition}  Let $N$ be a $C$-module which is free over $k[s,t]$ satisfying 
$\sN^{\odd}\cong \sN^{\ev}\otimes_{\sC^{\ev}}\sC^{\odd}$. Let $\TT(N)$ be the complex constructed in Theorem \ref{construction of TN}
whose terms are described by cohomology groups of $\sA_{E}=\sN^{\ev}$ and $\sB=\sN^{\odd}$ and their twists. The cokernel $G_{X}$ of the map
\[
H^1(\sB_E(-1)) \otimes P_X(-1) \rTo H^1(\sA_E) \otimes P_X,
\]
which is a component of   the differential $F_{-1}\to F_{0}$ of $\TT(N)$, is an Ulrich module if and only if $H^1(\sB_E)$ and  $H^0(\sB_E)$ vanish.
\end{theorem}

\begin{proof} If $G_X$ is an Ulrich $P_X$-module, then it is its own 
MCM approximation. Hence the Tate resolution of $G_X$ has non-overlapping strands so $H^1(\sB_E)$ and  $H^0(\sB_E)$ vanish.  

Conversely, if these groups vanish then $G_X$ is a MCM module over $P_X$ with a linear $P_X$-resolution, and from the form of the complex $\TT(N)$ we see that
$H^0(\sA_E)$ and all terms to the left of it in the lower row
must also vanish.
To show that  $G_X$ is an Ulrich module
we must prove that $G_X$ has a linear resolution as a~$P$-module.

We first make the form of the $P_X$-resolution more explicit. The cohomological vanishing \mbox{$h^0(\sB_E)=h^1(\sB_E)=0$} implies that $\pi_* \sB_E = \sO_{\PP^1}(-1)^{2 r}$,
where $r= \rank \sB_E= \rank \sA_E$. Since $\sB(-p) \cong \sA$ we have
$\deg \sA_E= \deg \sB_E - r$.  Thus $H^0(\sA_E)=0$ and, by the Riemann-Roch formula, $h^1(\sA_E) =r$. 
The form of the Tate resolution implies that the bundle $\pi_* \sA_E$ splits into a direct sum of  copies of $\sO_{\PP^1}(-1)$ and 
$\sO_{\PP^1}(-2)$. Indeed, there cannot be any summands of the form $\sO_{\PP^1}(-d)$ with $d\le -3$
because there are no nonzero maps to this sheaf from $\pi_* \sB_E(-1)=\sO_{\PP^1}(-2)^{2 r}$.
Hence
\[
\pi_* \sA_E = \sO_{\PP^1}(-1)^{r}\oplus \sO_{\PP^1}(-2)^{r}.
\]
Since $\pi_* \sB_E(-1) = \sO_{\PP^1}(-2)^{2 \rank \sB_E}$ we see that $G_X$ is defined by an $r \times 2r$ matrix of linear forms and  the $P_X$-free resolution of $G_X$ has the form
\[
\cdots \rTo P_X^{(i+1)r}(-i) \rTo \cdots \rTo P_X^{2r}(-1) \rTo^{\phi_{1}} P_X^r \rTo G_X \rTo 0.
\]

We can now show that $G_X$ has linear resolution as a $P$-module.
 Since $G_X$ is maximal Cohen--Macaulay module over $P_X$, this statement
 can be checked after factoring out a maximal $P_X$-regular sequence $z$ of linear forms in $P$.
Note that $P_X/zP_X$ has Hilbert function $1,2,1$. The sequence $z$ is also a regular sequence on $G_X$ because
$G_X$ is a maximal Cohen--Macaulay module.
 From the resolution of $G_X$ over $P_X$ we see
  that the values of the Hilbert function of $G_X/zG_X$ are $r, 0,0,\ldots$; that is, 
 $G_X/zG_X \cong k^{r}$. As a module over $P/zP$ this has a linear resolution, and thus $G_X$ has
 a linear resolution as a $P$-module. 
 Thus $G_X$ is an Ulrich $P_X$-module. 
 \end{proof}

\begin{remark}
 The proof shows in particular that, the matrix 
\[
P^{2r}(-1) \rTo^{\phi_{1}} P^r
\]
obtained by regarding the linear $P_X$-presentation of $G_X$ as a matrix over $P$
is a presentation matrix of $G_X$ as a $P$-module.
\end{remark}
Using the Morita equivalence between the hyperelliptic curve $E$ and the Clifford algebra $C$ we can make this more precise. Recall that a bundle $\sB$ on $E$ has the Raynaud property if $H^0(C, \sB) = H^1(C,\sB)=0$. We are now ready
to prove  parts of Theorem ~\ref{main theorem} from the introduction, which we repeat for the reader's convenience.

\begin{theorem}\label{main1}
There is a $1-1$ correspondence between Ulrich bundles on the smooth complete intersection of two quadrics $X\subset \PP^{2g+1}$ and  bundles with the Raynaud property on the corresponding hyperelliptic curve $E$ of the form
$\sG \otimes \sF_{U}$.
The Ulrich bundle corresponding to a rank $r$ vector bundle $\sG$ has rank $r2^{g-2}$. 

If $\sL$ is a line bundle on $E$ then $\sL \otimes \sF_{U}$ does not have the Raynaud property, so the minimal possible rank of an Ulrich sheaf on $X$ is $2^{g-1}$, and Ulrich bundles of  rank $2^{g-1}$ exist.
\end{theorem}

\begin{proof}
Let $p \in E$ be a ramification point. Consider  $\sB=\sG\tensor \sF_U$, $\sA=\sG(-p) \tensor \sF_U$ and the Clifford module
$N=\oplus_{i} H^0(\sA(ip))$. By Theorem \ref{UlrichCondition} $\TT(N)$ is the Tate resolution of the Ulrich module 
$G_X =\coker\left( H^1(\sB_E(-1)) \otimes P_X(-1) \to H^1(\sA_E) \otimes P_X\right)$
if and only if $H^0(\sB)=H^1(\sB)=0$. 
 If 
$r= \rank \sG$
and the condition is satisfied then the corresponding Ulrich module $G_X$ on $X$ has rank $G_X = r 2^{g-2}$
since the number of generators of $G_X$ is $\rank (\sG \otimes \sF_U)=r 2^g$.

Conversely, suppose that $M$ is an Ulrich module on $P_X$, and let $N= \Ext_{P_X}(M,k)$. This is a $C$-module,
and thus an $R_E$-module which is a free $k[s,t]$-module
by the Eisenbud--Shamash construction \cite[Theorem~7.2]{Eisenbud80}. The odd part of its sheafification is thus of the form
 $\sN^{\odd} = \sG \otimes_{\sO_E} \sF_U$ for some a vector bundle $\sG$ by Corollary~\ref{morita}, the Morita theorem. By 
 Theorem \ref{UlrichCondition} $\sG \otimes_{\sO_E} \sF_U$ has the Raynaud property.

An Ulrich module of rank $2^{g-2}$ would correspond to a line bundle $\sL$ on $E$
such that $\sL\otimes \sF_U$ has vanishing cohomology.  By Corollary \ref{ReidsTheorem}, $\sL \otimes \sF_U = \sF_{U'}(mp)$ for some
maximal isotropic plane $U'$ and some integer $m$. Thus $\TT(N)^*$ would be the Tate resolution of $P_{U'}$
up to shift. But 
$\TT(P_{U'})$ has overlapping strands (in fact $P_{U'}$ is not a MCM $P_X$-module).

The existence of Ulrich bundles of rank $2^{g-1}$ is proven in Section~\ref{sec:Ulrich}.
\end{proof}

\begin{proposition}\label{congruenceCondition} Ulrich bundles of rank $r2^{g-2}$ on a smooth complete intersections of two quadrics in $\PP^{2g+1}$ do not exist
if $r\cdot g \equiv 1 \mod 2$.
\end{proposition}

\begin{proof} If $\sG$ is a  vector bundle on $E$ of rank $r$ and degree $d$ then
\[
\deg (\sG\otimes \sF_{U}) = \deg \sG \rank \sF_{U}+\rank \sG\deg \sF_{U}= d2^{g}+rg2^{g-1}
\] 
by Proposition \ref{degreeOfFu} and
\[
\chi(\sG\otimes \sF_{U}) = \deg (\sG\otimes \sF_{U})+\rank (\sG\otimes \sF_{U})(1-g) 
=d2^{g}+rg2^{g-1}+r2^{g}(1-g)
\]
by Riemann-Roch.
Thus $\chi(\sG\otimes \sF_{U})=0$ implies $r\cdot g\equiv 0 \mod 2$.
\end{proof}

For small $g$ we constructed Ulrich bundles of rank $2^{g-1}$ from sufficiently general 
rank $2$ bundles $\sG$ on $E$ with our Macaulay2 package \cite{EKS}. 
Consider the  direct sum  $\sG_0=\sL_0 \oplus \sL_g$ of two general line bundle $\sL_i$ of degree $i$. 
In case of $g=3$ the cohomology table of the bundle
$\sG_0 \otimes \sF_U$ is the sum of two tables, one of which we displayed in Example \ref{g=3example bis} in case of $g=3$.  
The other is a shifted version of that table. 
 
So in case of $g=3$ the cohomology table of $\sG_0\otimes \sF_U$ 
has shape
\[
\begin{matrix}
\cdots&64&48&33&21&12&5&1& & \\
&&1&5&12&21&33&48&64 &\cdots.
\end{matrix}
\]
If for a general extension $0 \to \sL_0 \to \sG \to \sL_3 \to 0$ 
the connecting homomorphisms are of maximal rank, then the cohomology table of $\sG\otimes \sF_U$ has the form
\[
\begin{matrix}
\cdots&64&48&32&16& & & & & \\
&& & & &16&32&48&64 &\cdots.
\end{matrix}
\]
and $\sG$ gives rise to an Ulrich bundle of rank $2\cdot 2^{g-2}$.
 In special cases, for small $g$ we verified that this does occur with Macaulay2 \cite{M2} using our package \cite{EKS}.
 With the same idea we constructed Ulrich bundles of rank $3\cdot 2^{g-2}$ in special cases for $g=2$.

However we were not able to control the cohomology of 
$\sG \otimes \sF_U$ theoretically well enough to prove the existence of rank $2^{g-1}$ Ulrich bundle for every $X$.

\section{Ulrich bundles of rank $2^{g-1}$}\label{sec:Ulrich}

In this section we prove that a smooth complete intersection of two quadrics in $\PP^{2g+2}$, and therefore also in $\PP^{2g+1}$,  carries an Ulrich bundle of rank 
$2^{g-1}$. Our construction uses the construction of Ulrich bundles on a single quadric by  Kn\"orrer, which we now review.

\begin{theorem}[\emph{cf.}~\cite{Knorrer}]
The quadric $q_n=\sum_{i=0}^n x_iy_i$ has the matrix factorization $(\varphi_n,\psi_n)$ of size $2^{n}$ defined recursively by $\varphi_0=(x_0),\psi_0=(y_0)$ and
\[
\varphi_n = \begin{pmatrix}x_n & \varphi_{n-1} \cr
\psi_{n-1} & -y_n \end{pmatrix}, \quad
\psi_n = \begin{pmatrix}y_n & \varphi_{n-1} \cr
\psi_{n-1} & -x_n \end{pmatrix}
\]
for $n\ge 1$.                                  
\end{theorem}      

Let $(A,B) =(\varphi_n,\psi_n)$ and consider the matrix factorizations
\[
(A(x,y), B(x,y))\quad \text{and}\quad (A(v,w),B(v,w))
\]
of $q(x,y)=\sum_{i=0}^n x_iy_i$ and $q(v,w)=\sum_{i=0}^n v_iw_i$ respectively over the ring $P:=k[x|y,v|w]$,
where $x|y$ denotes the catenation $x_0,\ldots, x_n, y_0,\ldots,y_n$ and similarly for $v|w$.

\begin{proposition}\label{special product}  Let 
\[
\widetilde q(v,w,x,y)=\sum_{i=0}^n (x_iw_i+y_iv_i) = (v|w)\cdot (y|x).
\]
There is an identity
\[
\begin{pmatrix}A(x,y)& A(v,w)\end{pmatrix} 
\begin{pmatrix}B(v,w)\cr B(x,y)\end{pmatrix}
= \widetilde q(v,w,x,y) \id_{2^n}.
\]
\end{proposition}

\begin{proof}
Since $A(x,y)+A(v,w)=A(x+v,y+w)$ and $B(x,y)+B(v,w)=B(x+v,y+w)$ we have
\[A(x+v,y+w)B(x+v,y+v)=  q(x+v,y+w) \id_{2^n}.\]
The mixed terms give
\[A(x,y)B(v,w)+A(v,w)B(x,y)=\widetilde q(v,w,x,y)\id_{2^n}.\qedhere\]
\end{proof}

Thus if we restrict the matrices in Proposition~\ref{special product} to an isotropic subspace $\Sigma$ of $\widetilde q$ we get 
a complex and we will see that, for a sufficiently general choice of the isotropic subspace, the restriction to $\Sigma$ is a minimal free resolution
of an Ulrich module over $P_{\Sigma}$.

To define the isotropic subspace, let $\Lambda$ be a skew-symmetric $2(n+1)\times 2(n+1)$ matrix of scalars, and set
\[
G_\Lambda=\begin{pmatrix}0 & \id_{n+1} \cr \id_{n+1}&0 \end{pmatrix} \Lambda.
\]
We have
\[
(x|y)G_\Lambda \cdot (y|x) = (y|x)\Lambda \cdot (y|x) =0
\]
and thus the equation $(v|w)=(x|y)G_\Lambda$ defines an isotropic subspace of  
$\widetilde q(v,w,x,y)$. 

The matrices  
\[
A_1 =A(x,y),\ B_1= B(x,y)\quad
\text{and}\quad
A_2 = A((x|y)G_\Lambda),\ B_2 = B((x|y)G_\Lambda)
\]
define matrix factorizations of $q_1 = q(x,y)$ and $q_2 = q((x|y)G_\Lambda))$.
Let 
\[
A_\Lambda=A_1|A_2
\]
be the concatenation, which is a  $2^n\times 2^{n+1}$ matrix
in the $2n+2$ variables $x_0,\ldots,y_n$. 

\begin{theorem}\label{specialUlrich} For a general choice of $\Lambda$ the 
 ring $k[x_0,\ldots,y_n]/(q_1, q_2)$ is a complete
intersection with isolated singularities and
\[
M_\Lambda := \coker A_\Lambda
\]
is an Ulrich module of rank $2^{n-2}$ over this ring.
\end{theorem}

\begin{proof}
 Set $P=k[x_0,\ldots,y_n]$. For each $\Lambda$ we have maps
\[
\begin{diagram}
0&\lTo &M_\Lambda &\lTo &P^{2^n} &\lTo^{\begin{pmatrix}
A_1 & A_2
\end{pmatrix}
} & P^{2^{n+1}}(-1)
&\lTo^{\begin{pmatrix}
B_2\\B_1
\end{pmatrix}
}  
&P^{2^n}(-2) &\lTo & 0
\end{diagram}.
\]
By our choice of $A_2$ and $B_2$ this is a complex.

 We claim that for a general choice of $\Lambda$ the ideal $(q_1,q_2)$ is a prime ideal of
 codimension $2$ with isolated singularities. It suffices to prove this 
for a particular choice of $\Lambda$. 

We will actually prove the result for matrices $\Lambda$ of the form  
\[
\Lambda=\begin{pmatrix}0 & D \\ -D&0 \end{pmatrix}
\]
where $D$ is a diagonal matrix with entries $d_i$
 such that 
\[
d_0,\ldots,d_n,-d_0,\ldots,-d_n
\]
are $2(n+1)$ different values. In this case
\[
G_\Lambda=\begin{pmatrix} -D & 0 \cr 0& D \end{pmatrix},\quad
A_\Lambda=(A(x_0,\ldots,x_n,y_0,\ldots,y_n)|A(-d_0x_0,\ldots-d_nx_n,d_0y_0,\ldots,d_ny_n),
\]
and
\[
q_2 = q_1(-d_0x_0, \dots, -d_nx_n, d_0y_0, \dots, d_ny_n) = -\sum_{i=0}^n d_i^2x_iy_i.
\]

We will now show that $V(q_1, q_2)$ is singular precisely at the coordinate points. The jacobian matrix of $(q_1,-q_2)$ is
\[
\begin{pmatrix} 
y_0 & y_1 & \ldots & y_n & x_0 & \ldots & x_n \cr
d_0^2y_0 & d_1^2y_1 & \ldots & d_n^2y_n & d_0^2x_0 & \ldots & d_n^2x_n 
\end{pmatrix}
\]
 The squares $d_0^2, \ldots,d_n^2$ are pairwise distinct, since $d_0,\ldots,d_n,-d_0,\ldots, -d_n$ are $2(n+1)$ distinct values by assumption.
 Thus the zero locus of the ideal of $2\times 2$ minors of the jacobian matrix is the union of the $n+1$ lines $L_i= V(\bigcup_{j \not=i}\{x_j,y_j\})$
 defined by those linear combinations of the two rows that do not consist of independent linear forms. These lines intersect $V(q_1,q_2)$ in the $2n+2$ coordinate points. It follows that $(q_1,q_2)$ has codimension 2 and isolated singularities, and thus is prime.
 
Since each $q_i$ is prime and $A_i$ is part of a matrix factorization of $q_i$, the determinant of $A_i$ is a power of $q_i$. 
Thus if $\Lambda$ is general,
 the maximal minors of $A_\Lambda$ 
generate an ideal of codimension at least 2, and similarly for $B_\Lambda$ so the complex is exact by \cite{Buchsbaum-Eisenbud}.

We conclude that
\[\ann\; M_\Lambda =(q_1,q_2)\]
since any element of $\ann\; M_\Lambda \setminus (q_1,q_2)$ would lead to a support of codimension at least $3$. Thus $M_\Lambda$ is an Ulrich module over the ring $P/(q_1, q_2)$
 and the degree of $M_\Lambda$ is $2^n$, so the rank of $M_\Lambda$ as  an $P/(q_1,q_2)$-module is~$2^{n-2}$.
 \end{proof}

\begin{theorem}\label{2n+1} Let $k$ be an algebraically closed field of $\chara k \not= 2$, and $X \subset \PP^{2n}$ be a smooth complete 
intersection of two quadrics. Then $X$ carries an Ulrich bundle of rank $2^{n-2}$.
\end{theorem}

\begin{corollary}\label{reduction}
Let $k$ be an algebraically closed field of $\chara k \not= 2$, and $X \subset \PP^{2g+1}$ be a smooth complete 
intersection of two quadrics. Then $X$ carries an Ulrich bundle of rank $2^{g-1}$.
\end{corollary}

\begin{proof}[Proof of Corollary~\ref{reduction}] Any smooth complete intersection in $\PP^{2g+1}$ is a hyperplane section of a smooth complete intersection in
$\PP^{2g+2}$. Taking $n=g+1$, the restriction of the Ulrich module constructed in Theorem~\ref{2n+1} is an Ulrich module of rank
$2^{g-1}$.
\end{proof}

\begin{proof}[Proof of Theorem~\ref{2n+1}] We obtain an Ulrich module on some smooth complete intersection by restricting $M_\Lambda$ from above to a general hyperplane $H=\PP^{2n} \subset \PP^{2n+1}$. The intersection will be smooth because $V(q_1,q_2)$ has only isolated singularities. To prove that every smooth complete intersection carries an Ulrich module we need additional
arguments. The complete  intersection $V(q_1',q_2')$ of two quadrics in $\PP^{2n}$ is smooth if and only if the discriminant
\[
f= \det \hess (sq_1'+q_2') \in k[s]
\]
of the pencil has $2n+1$ distinct roots, and in that case $q_1'$ and $q_2'$ can be simultaneously diagonalized by the
argument given at the beginning of Section~\ref{pencil section}.
Thus it suffices to construct an Ulrich module $M'$ on a the complete intersection $V(q_1',q_2')$ whose discriminant has any given set of $2n+1$ distinct roots. In the proof of Theorem \ref{specialUlrich} we constructed an Ulrich module for
$q_1=\sum_{i=0}^n x_iy_i$ and $q_2=-\sum_{i=0}^n d_i^2x_iy_i$ for distinct values $d_0^2,\ldots,d_n^2$. 
Since $k$ is algebraically closed there exists an Ulrich module for $V(\sum_{i=0}^n x_iy_i, \sum_{i=0}^{n} a_i x_iy_i)$
for every tuple of  distinct values $a_0,\ldots, a_{n}$. The corresponding Hessian is
\[
H = \begin{pmatrix}
0 & D' \cr
D' & 0 
\end{pmatrix} \hbox{ with a diagonal matrix } D'=
\begin{pmatrix} 
s+a_0 & & \cr
 & \ddots & \cr
 & & s+a_n
 \end{pmatrix}.
\]
We restrict the quadrics to the subspace
generated by the columns of the $(2n+2) \times (2n+1)$ matrix
of 
\[
B=\begin{pmatrix}
 1 &  & 0 \cr
  & \ddots & \cr
  0 & & 1 \cr
  b_0 & \dots & b_{2n}
\end{pmatrix}.
\]
Setting $\ell_i = s+a_i$
the Hessian of the restricted pencil is 
\[
B^t H B=\begin{pmatrix}
&&& \ell_nb_0 & \ell_0& &  \cr
& 0 & & \vdots & & \ddots & \cr
& & & \ell_nb_{n-1} &  & & \ell_{n-1} \cr
\ell_nb_0& \dots &\ell_nb_{n-1} & 2\ell_nb_{n} &\ell_nb_{n+1}  & \dots & \ell_nb_{2n} \cr
\ell_0& & & \ell_nb_{n+1} & && \cr
& \ddots & & \vdots & & 0 & \cr
&  &\ell_{n-1} & \ell_nb_{2n} & & & \cr
\end{pmatrix}.
\]
Direct computation shows that the determinant of this matrix is
\[f = (-1)^{n} 2  h \prod_{i=0}^{n}\ell_i=(-1)^{n} 2  h \prod_{i=0}^{n}(s+a_i) \]
with 
\[h = \sum_{i=0}^{n-1} (b_ib_{i+n+1} \prod_{j\neq i} (s+a_j)) -b_n \prod_{j\neq n}(s+a_j).\]
Since the coefficients of $ \prod_{j\neq i}(s+a_j)$ are the elementary symmetric functions $e_{i,k}$
on $\{a_0,\ldots,a_{n} \} \setminus \{a_i\}$, we obtain
\begin{equation}\label{1} h=(b_0b_{n+1}, \ldots,  b_{n-1}b_{2n}, -b_n) E \begin{pmatrix}s^n \cr \vdots \cr s \cr 1\end{pmatrix}
\end{equation}
where $E=(e_{i,k})_{\substack{i=0,\ldots,n \\ k=0,\ldots, n}}$. We claim that
\[\det E= \prod_{0 \le i < j \le n} (a_i-a_j).\]
Regarding the $a_i$'s as variables, we see that  $\det E \in k[a_0,\ldots,a_{n}]$ is not identically zero, because the term $\prod_{i=0}^{n-1} a_i^{n-i}$
occurs precisely once in the determinant as the product of the leading terms $1,a_0,a_0a_1,\ldots,a_0a_1\dots a_{n-1}$ of the diagonal entries. On the other hand 
$(a_i-a_j)$ is a factor of  $\det E \in k[a_0,\ldots,a_{n}]$ because if $a_i=a_j$ then the  matrix $E$ has two equal rows.
So these linear forms are factors of $\det E \in k[a_0,\ldots,a_{n}]$, and their product coincides with $\det E$ for degree reasons and by comparing the leading term.

Thus if the $a_i$ are distinct, then $E$ is invertible, and every polynomial $h$ of degree $n$ in $k[s]$ can be represented
in the form (\ref{1}).
In particular, we can choose $b_0,\dots, b_{2n} \in k$ such that the discriminant
$f$ is equal to $\prod_{i=0}^{n}(s+a_i) \prod_{i=1}^n (s+c_i)$ for any $2n+1$ distinct non-zero values $a_0,\ldots,a_{n},c_1,\ldots,c_n \in k$.  
A smooth complete intersection 
of 2 quadrics in $\PP^{2n}$ is determined up to projective equivalence by the $2n+1$ distinct roots of its discriminant, this concludes the proof.
\end{proof}


\newcommand{\etalchar}[1]{$^{#1}$}
\providecommand{\bysame}{\leavevmode\hbox to3em{\hrulefill}\thinspace}
\providecommand{\MR}{\relax\ifhmode\unskip\space\fi MR }
\providecommand{\MRhref}[2]{%
  \href{http://www.ams.org/mathscinet-getitem?mr=#1}{#2}
}
\providecommand{\href}[2]{#2}


\begin{thebibliography}{BEH{\etalchar{+}}23++}

\bibitem[AB89]{Auslander-Buchweitz}
M.~Auslander and R.-O. Buchweitz, \emph{The homological theory of maximal
  {Cohen}-{Macaulay} approximations}, M{\'e}m. Soc. Math. Fr., Nouv. S{\'e}r.
  \textbf{38} (1989), 5--37.

\bibitem[Bas68]{Bass}
H.~Bass, \emph{Algebraic {$K$}-theory}, W. A. Benjamin, Inc., New
  York-Amsterdam, 1968.

\bibitem[BEH{\etalchar{+}}24]{BEHLV}
A.~Beauville, A.~Etesse, A.~H\"oring, J.~Liu, and C.~Voisin, \emph{Symmetric
  tensors on the intersection of two quadrics and lagrangian fibration}, Moduli~\textbf{1} (2024), Article ID e4.

\bibitem[BE73]{Buchsbaum-Eisenbud}
D.~A. Buchsbaum and D.~Eisenbud, \emph{What makes a complex exact?}, J. Algebra
  \textbf{25} (1973), 259--268.

\bibitem[BEH87]{BEH}
R.-O. Buchweitz, D.~Eisenbud, and J.~Herzog, \emph{Cohen-{M}acaulay modules on
  quadrics}, in: \emph{Singularities, representation of algebras, and vector bundles} ({L}ambrecht, 1985), Lecture Notes in Math., vol.~\textbf{1273}, Springer, Berlin, 1987, pp.~58--116.

\bibitem[Buc21]{Buchweitz}
R.-O. Buchweitz, \emph{Maximal {Cohen}-{Macaulay} modules and {Tate}
  cohomology. {With} appendices by {Luchezar} {L}. {Avramov}, {Benjamin}
  {Briggs}, {Srikanth} {B}. {Iyengar} and {Janina} {C}. {Letz}}, Math. Surv.
  Monogr., vol.~\textbf{262}, Providence, RI: American Mathematical Society, 2021.

\bibitem[CKL21]{ChoKimLee}
Y.~Cho, Y.~Kim, and K.-S. Lee, \emph{Ulrich bundles on intersections of two
  4-dimensional quadrics}, Int. Math. Res. Not. \textbf{2021} (2021), no.~22,
  17277--17303.

\bibitem[DR76]{Desale-Ramanan}
U.~V. Desale and S.~Ramanan, \emph{Classification of vector bundles of rank 2
  on hyperelliptic curves}, Invent. Math. \textbf{38} (1976), 161--185.

\bibitem[Don80]{Donagi}
R.~Donagi, \emph{Group law on the intersection of two quadrics}, Ann. Sc. Norm.
  Super. Pisa, Cl. Sci., IV. Ser. \textbf{7} (1980), 217--239.

\bibitem[EFS03]{EFS}
D.~Eisenbud, G.~Fl{\o}ystad, and F.-O. Schreyer, \emph{Sheaf cohomology and
  free resolutions over exterior algebras}, Trans. Am. Math. Soc. \textbf{355}
  (2003), no.~11, 4397--4426.

\bibitem[Eis80]{Eisenbud80}
D.~Eisenbud, \emph{Homological algebra of a complete intersection, with an
  application to group representations}, Trans. Am. Math. Soc. \textbf{260}
  (1980), 35--64.

\bibitem[EKS22]{EKS}
D.~Eisenbud, Y.~Kim, and F.-O. Schreyer, \emph{Macaulay2 package ``{Clifford}
  algebra of a pencil of quadratic forms''}, Available at
  \url{https://macaulay2.com/doc/Macaulay2/share/doc/Macaulay2/PencilsOfQuadrics/html/index.html},
  2022.

\bibitem[ES03]{ES1}
D.~Eisenbud and F.-O. Schreyer, \emph{Resultants and {C}how forms via exterior
  syzygies}, J. Amer. Math. Soc. \textbf{16} (2003), no.~3, 537--579, With an
  appendix by Jerzy Weyman.

\bibitem[ES21]{EisenbudSchreyer}
\bysame, \emph{Tate resolutions and {MCM} approximations}, in: \emph{Commutative algebra.
  150 years with Roger and Sylvia Wiegand.} Combined proceedings: second
  international meeting on commutative algebra and related areas, SIMCARA,
  Universidade de S\~ao Paulo, S\~ao Carlos, Brazil, July 22--26, 2019 and AMS
  special session on commutative algebra. In celebration of the 150th birthday
  of Roger and Sylvia Wiegand, University of Wisconsin-Madison, Wisconsin,
  September 14--15, 2019, Providence, RI: American Mathematical Society,
  2021, pp.~35--47.

\bibitem[FL10]{MR2639318}
G.~Farkas and K.~Ludwig, \emph{The {K}odaira dimension of the moduli space of
  {P}rym varieties}, J. Eur. Math. Soc. (JEMS) \textbf{12} (2010), no.~3,
  755--795.

\bibitem[Gan59]{Gantmacher}
F.~R. Gantmacher, \emph{The theory of matrices. {V}ols. 1, 2}, Chelsea
  Publishing Co., New York, 1959, Translated by K. A. Hirsch.

\bibitem[GS]{M2}
D.~R. Grayson and M.~E. Stillman, \emph{Macaulay2, a software system for
  research in algebraic geometry}, Available at \url{https://macaulay2.com}.

\bibitem[Haa91]{Haag}
U.~Haag, \emph{Diskriminantenalgebren quadratischer {Formen}. ({Discriminant}
  algebras of quadratic forms)}, Arch. Math. \textbf{57} (1991), no.~6,
  546--554.

\bibitem[Jac46]{Jacobi}
C.~G.~J. Jacobi, \emph{{\"U}ber eine neue {Methode} zur {Integration} der
  hyperelliptischen {Differentialgleichungen} und {\"u}ber die rationale {Form}
  ihrer vollst{\"a}ndigen algebraischen {Integralgleichungen}}, J. Reine
  Angew. Math. \textbf{32} (1846), 220--226.

\bibitem[Jac80]{Jacobson}
N.~Jacobson, \emph{Basic {Algebra} {II}}, San {Francisco}: {W}. {H}. {Freeman}
  and {Company}, {XX}, 1980.

\bibitem[Kap89]{Kapranov}
M.~M. Kapranov, \emph{On the derived category and {K}-functor of coherent
  sheaves on intersections of quadrics}, Math. USSR, Izv. \textbf{32} (1989),
  no.~1, 191--204.

\bibitem[Kn{\"o}87]{Knorrer}
H.~Kn{\"o}rrer, \emph{Cohen-{Macaulay} modules on hypersurface singularities.
  {I}}, Invent. Math. \textbf{88} (1987), 153--164.

\bibitem[Kuz08]{Kuznetsov}
A.~Kuznetsov, \emph{Derived categories of quadric fibrations and intersections
  of quadrics}, Adv. Math. \textbf{218} (2008), no.~5, 1340--1369.

\bibitem[Mum84]{Mumford-Tata}
D.~Mumford, \emph{Tata lectures on theta. {II}: {Jacobian} theta functions and
  differential equations.} {With} the collaboration of {C}. {Musili}, {M}.
  {Nori}, {E}. {Previato}, {M}. {Stillman}, and {H}. {Umemura}, Prog. Math.,
  vol.~\textbf{43}, Birkh{\"a}user, Cham, 1984.

\bibitem[New68]{Newstead}
P.~E. Newstead, \emph{Stable bundles of rank 2 and odd degree over a curve of
  genus 2}, Topology \textbf{7} (1968), 205--215.

\bibitem[PP05]{PolishchukPositselski}
A.~Polishchuk and L.~Positselski, \emph{Quadratic algebras.}, Univ. Lect. Ser.,
  vol.~\textbf{37}, Providence, RI: American Mathematical Society, 2005.

\bibitem[Rei72]{ReidThesis}
M.~Reid, \emph{The complete intersection of two or more quadrics}, Thesis,
  University of Cambridge, 1972.

\bibitem[Sj{\"o}76]{Sjodin}
G.~Sj{\"o}din, \emph{A set of generators for {{\(\text{Ext}_r(k,k)\)}}}, Math.
  Scand. \textbf{38} (1976), 199--210.

\bibitem[Tat57]{Tate}
J.~Tate, \emph{Homology of {N}oetherian rings and local rings}, Illinois J.
  Math. \textbf{1} (1957), 14--27.

\end{thebibliography}
\end{document}